\documentclass[reqno,11pt,letterpaper]{amsart}
\usepackage{amsxtra,amscd,amsthm,amsmath,amssymb}
\usepackage{longtable} 
\usepackage{enumerate}
\usepackage{verbatim}
\usepackage{graphicx}
\usepackage{xspace}
\usepackage{ifthen}
\usepackage{tabularx}

\newcommand{\quash}[1]{}  

\quash{ 
 \textwidth 5.56in
  \textheight 8.41in
 \oddsidemargin 0.53in
 \evensidemargin 0.53in 
 \topmargin -0.10in
   
}

\usepackage{amsthm,amsfonts,amssymb,amsmath,amsxtra}
\usepackage[all]{xy}
\usepackage{color}
\usepackage{mathrsfs}

\makeatletter
\renewcommand*\env@matrix[1][*\c@MaxMatrixCols c]{%
  \hskip -\arraycolsep
  \let\@ifnextchar\new@ifnextchar
  \array{#1}}
\makeatother

\usepackage{flexisym}
\usepackage[utf8]{inputenc}
\usepackage{amsfonts}
\usepackage{multicol}
\usepackage{collectbox}

\usepackage{flexisym}
\usepackage{amsfonts}
\usepackage{amsmath}

 \usepackage{hyperref}

\usepackage{multicol}
\usepackage{collectbox}


\usepackage{amsthm}

\theoremstyle{plain}
\numberwithin{equation}{subsection}
\newtheorem{theorem}[equation]{Theorem}

\newtheorem{proposition}[equation]{Proposition}
\newtheorem{lemma}[equation]{Lemma}

\newtheorem{Remark}[equation]{Remark}

\newtheorem{conjecture}[equation]{Conjecture}
\newtheorem{corollary}[equation]{Corollary}

\theoremstyle{remark}

\theoremstyle{plain}
\theoremstyle{definition}\newtheorem*{notation}{Notation}

\newcommand{\nc}{\newcommand}
\nc{\on}{\operatorname}

\renewcommand{\phi}{\varphi}

\newcommand{\bbA}{{\mathbb A}}

\newcommand{\bbF}{{\mathbb F}}

\newcommand{\bbQ}{{\mathbb Q}}

\newcommand{\bbZ}{{\mathbb Z}}

\newcommand{\calF}{{\mathcal F}}
\newcommand{\calG}{{\mathcal G}}

\newcommand{\calS}{{\mathcal S}}

\newcommand{\ov}{\overline}

\nc{\al}{{\alpha}} \nc{\be}{{\beta}}

\nc{\ve}{{\varepsilon}} \nc{\Ga}{{\Gamma}}

\nc{\La}{{\Lambda}}

\def\0{\circ}

\newcommand{\cal}{\mathcal}

\newcommand{\A}{{\mathcal A}}

\newcommand{\Q}{{\mathbb Q}}

\newcommand{\Hom}{{\rm Hom}}

\newcommand{\Z}{{\mathbb Z}}

\newcommand{\M}{{\mathcal M}}

\newcommand{\K}{{\mathcal K}}

\def\thfill{\null\nobreak\hfill}

\def\endproof{\thfill\vbox{\hrule
  \hbox{\vrule\hbox to 5pt{\vbox to 5pt{\vfil}\hfil}\vrule}\hrule}}

\newcommand{\lan}{\langle}
\newcommand{\ran}{\rangle}

\newcommand{\BL}{{\mathbb {L}}}

\newcommand{\bb}{\mathbb}

\newcommand{\op}{\operatorname}
\newcommand{\quis}{\stackrel{\sim}{\to}}

\newcommand{\To}{\longrightarrow}
\newcommand{\xto}{\xrightarrow}

\renewcommand{\cal}{\mathcal}
\newcommand{\indlim}{\varinjlim}

\quash{
\documentclass{article}
\usepackage{amssymb}
\usepackage{amsfonts}
\usepackage{amsmath}

\setcounter{MaxMatrixCols}{10}

\newtheorem{theorem}{Theorem}

\newtheorem{corollary}[theorem]{Corollary}

\newtheorem{lemma}[theorem]{Lemma}

\newtheorem{proposition}[theorem]{Proposition}

\newenvironment{proof}[1][Proof]{\noindent\textbf{#1.} }{\ \rule{0.5em}{0.5em}}
}

\newcommand{\QTP}[1]{}

\newcommand{\RK}{\mathrm{K}}
\newcommand{\RG}{\mathrm{G}}
\newcommand{\RF}{\mathrm{F}}
\newcommand{\RNF}{\mathrm{NF}}
\newcommand{\RNK}{\mathrm{NK}}
\newcommand{\NK}{\mathrm{NK}}

\newcommand{\BZ}{\mathbb{Z}}

\newcommand{\HC}{\mathrm{HC}}

\newcommand{\prolim}{\varprojlim}

\newcommand{\wh}{\widehat}

\newcommand{\BF}{{\mathbb F}}

\newcommand{\rSK}{{\mathrm{SK}}}

\begin{document}

\title{The  Nil $\RK$-groups of finite groups}

 
 \author[T. Chinburg]{T. Chinburg}

\author[G. Pappas]{G. Pappas}

\author[M. Taylor]{M. J. Taylor  \\ \\ { {\scriptsize with  an Appendix by}} M. Morrow}

 \thanks{T.~C. was supported in part by Simons Foundation grant No. MP-TSM-00002279 and by NSF FRG grant DMS-2411702. G.~P. was supported in part by Simons Foundation grant SFI-MPS-TSM-00013296. M.~T. was supported by a grant from Merton College, Oxford.
M.~M was supported in part by the European Research Council (ERC) under the European Union’s Horizon 2020 research and innovation programme (grant agreement No. 101001474).}

\address{  Dept. of
Mathematics\\
Univ. of Pennsylvania\\
Philadelphia\\
PA 19104\\
USA} \email{ted@upenn.edu}

\address{ 
Dept. of
Mathematics\\
Michigan State
Univ.\\
E. Lansing\\
MI 48824\\
USA}
\email{pappasg@msu.edu}

\address{ Merton College\\ University of Oxford\\ OX1 4JD\\ U.K. }
\email{martin.taylor@merton.ox.ac.uk}
 
\address{ CNRS \& Laboratoire de Mathématiques d’Orsay\\
Faculté des Sciences d’Orsay, Université Paris-Saclay\\
F-91405 Orsay Cedex\\
France}
\email{matthew.morrow@universite-paris-saclay.fr}

\date{\today}

\begin{abstract}

We show that for every finite group $G$ and every integer $n$, the Nil group $\RNK_n(\BZ[G])$ of the integral group ring
 of $G$ has finite exponent, and give a bound depending on $n$ and the order $|G|$. This bound is explicit when 
  every prime divisor of $|G|$ is at least $3+n/2$.
 More generally, our finite exponent result applies to $\RNK_n(R[G])$   whenever $R$ is a regular, torsion-free, Noetherian commutative ring such that $R/p$ is regular for every prime $p$ that divides $|G|$.
 In the  Appendix, M. Morrow 
provides an alternative approach and also shows that 
the Nil group $\RNK_n(X)$ of an excellent Noetherian scheme $X$, with $X[1/p]$ regular,
is annihilated by a finite power of $p$, provided $X$ admits a suitable resolution of singularities. 
By extending his argument to certain noncommutative rings, we also show that, for $R$ as above and $\alpha$ any automorphism of $G$, the Farrell Nil groups $\RNK_n(R[G],\alpha)$ have finite exponent. As a consequence, for every virtually cyclic group $\Gamma$, the group $\RK_n(\BZ[\Gamma])$ is a direct sum of a finitely generated abelian group and an infinite countable direct sum of copies of a finite abelian group.
\end{abstract}
 
 \maketitle

\tableofcontents
 
   \maketitle

\date{\today}

 \section{Introduction} 
 
 \subsection{}  
 
The Nil $\RK$-groups of finite groups are among the most ubiquitous and poorly understood terms appearing in the algebraic $\RK$-theory of group rings. They arise already in the Bass–Heller–Swan decomposition for direct products with $\BZ$, and their twisted analogues occur naturally in the
 $\RK$-theory of group rings of virtually cyclic groups. These Nil groups are torsion and, when nontrivial, infinitely generated, but little else has been shown about their structure. The main result of this paper shows that, despite this nonfiniteness, the orders of their elements are uniformly bounded.
This leads to a description of the structure of the $\RK$-groups of the integral group ring  for every virtually cyclic group $\Gamma$; see Theorem \ref{cor2_Intro} below.

 \subsection{}   Let $G$ be a finite group and  denote by $\BZ[G]$ the group ring with integral coefficients. Fix an integer $n$. The subject of this paper is the Nil $\RK$-group $\RNK_n(\BZ[G])$, which can be defined as the kernel
 \[
 \RNK_n(\BZ[G])=\ker(\RK_n(\BZ[G][t])\xrightarrow{t\mapsto 0} \RK_n(\BZ[G])),
 \]
where $ \BZ[G][t]$ is the polynomial ring with coefficients $\BZ[G]$ and variable $t$. 
 More generally, for $\alpha: G\to G$ an automorphism, we also consider the Farrell Nil $\RK$-group  
  \[
\RNK_n(\BZ[G], \alpha):=   \ker(\RK_n(\BZ[G]_\alpha[t])\xrightarrow{t\mapsto 0} \RK_n(\BZ[G])).
  \]
 Here, $\BZ[G]_\alpha[t]$ is the skew polynomial ring with coefficients in $\BZ[G]$, variable $t$, and multiplication satisfying $t\cdot g = \alpha(g)t$, for all $g\in G$. When $\alpha={\rm id}$ one has $\RNK_n(\BZ[G], {\rm id})=\RNK_n(\BZ[G])$.

 The groups $ \RNK_n(\BZ[G])$ and $\RNK_n(\BZ[G], \alpha)$ are objects of classical interest with numerous connections to topology; see the recent survey by W.~L\"uck \cite{Luck}.
Indeed,   the Laurent polynomial ring $\BZ[G][t, t^{-1}]$ is the integral group ring of the direct product $ G\times \BZ$. Hence, the Nil group $\RNK_n(\BZ[G])$ is very closely related to the $\RK$-groups of the group ring of the ``infinite virtually  cyclic group" $\Gamma=G\times \BZ$ by the Bass--Heller--Swan formula; see Theorem \ref{thm:BHS}.
Similarly, $\RNK_n(\BZ[G], \alpha)$ features in the calculation of the $\RK$-groups of the group ring of the infinite virtually  cyclic group $\Gamma=G\rtimes_\alpha \BZ$.
In particular,  the first Nil group $ \RNK_1(\BZ[G],\alpha)$ relates to the Whitehead group ${\rm Wh}(\Gamma)$ of  $\Gamma=G\rtimes_\alpha \BZ$; see \cite{FaHs}.
The Whitehead group ${\rm Wh}(\Gamma)$ of a group $\Gamma$ is the receptacle of 
 ``Whitehead torsion": an element in  ${\rm Wh}(\Gamma)$ measures the obstruction for a homotopy equivalence between connected finite CW complexes with 
fundamental group $\Gamma$ to be a simple homotopy equivalence.
It also describes (by the $s$-cobordism theorem) the obstruction to exhibiting an $h$-cobordism between two manifolds in dimension $\geq 5$ with fundamental group $\Gamma$ as a product cobordism. With related topological applications in mind, 
such Nil groups have been studied by Farrell, Farrell--Hsiang,  Farrell--Jones and others.  For instance, the Farrell--Jones conjecture predicts, very roughly, that the $\RK$-groups of virtually cyclic groups provide the building blocks for the $\RK$-theory of all groups. This emphasizes the importance of understanding 
Nil groups like $ \RNK_n(\BZ[G])$ and $\RNK_n(\BZ[G], \alpha)$. For additional information, we refer the reader to \cite{Luck} and the references therein.

Nevertheless, the groups  $\RNK_n(\BZ[G], \alpha)$ have proven hard to understand and, so far, there have been only a few general results. It is known that
$ \RNK_n(\BZ[G], \alpha)$ is a torsion group. If nonzero, it is \emph{not} finitely generated: this follows by extensions of a construction of Farrell \cite{Fa},  see \cite{LPW} and the references therein. 

\subsection{}  We first discuss the ``untwisted case" $\RNK_n(\BZ[G])$, i.e. $\alpha={\rm id}$.
A few partial results determining these groups have been obtained in special  cases, mostly for $n=0$ or $1$. For example, for $p$ a prime, if $p^2\nmid |G|$ and $n\leq 1$,  then the localization $\RNK_n(\BZ[G])_{(p)}$ is trivial, see \cite{HLuck}. In particular, if the order $|G|$ is square-free, then $\RNK_1(\BZ[G])=(0)$, see also \cite{Harmon}. Weibel \cite{WNK} calculated $\RNK_0(\BZ[G])$ and $\RNK_1(\BZ[G])$ for  $G=C_4$, $C_2\times C_2$, and (almost for) $D_4$; these Nil groups are not trivial. 
 
By Weibel \cite{WMV} (Cor. \ref{cor:localization}, see also  Hambleton-L\"uck \cite{HLuck} and Kuku--Tang \cite{KukuTang}), 
every prime dividing the order of an element of $\RNK_n(\BZ[G])$ divides the order $|G|$ of $G$. 
However, there was no known uniform bound on the powers of these primes that could occur in the orders of elements. A basic question which had remained unanswered was whether there exists a \emph{finite exponent} for the  group $\RNK_n(\BZ[G])$, i.e. a positive integer which annihilates all  elements. In this paper, we give an affirmative answer.

\begin{theorem}\label{ZresultIntro}
For every finite group $G$ and $n\geq 0$, the group $\RNK_{n}(\BZ[G])$ has finite exponent.
\end{theorem}

Note that the groups $\RNK_n(\BZ[G])$ are trivial for $n<0$, by work of Farrell--Jones \cite{FaJo}. For $n=0$, Theorem \ref{ZresultIntro}  was already known by work of Connolly--Prassidis \cite{ConPra}. 

Combining Theorem \ref{ZresultIntro}  with work of Lafont--Prassidis--Wang \cite[Theorem C]{LPW}, which is based on Farrell's construction,  we obtain

 \begin{theorem}\label{thm:Countable}
For every finite group $G$ and $n\geq 0$, there is a finite abelian group $H$, whose exponent divides a power of  $|G|$, such that 
$\RNK_{n}(\BZ[G])$ is isomorphic to the infinite countable direct sum $ \oplus_\infty H$ of copies of $H$.
\end{theorem}

In particular, despite their nonfiniteness, these Nil groups have a strikingly rigid structure.

The proof brings recent advances in $\RK$-theory stemming from the introduction of trace methods to bear on this classical problem about integral group rings. Key tools
  are  pro-excision as developed by Geisser--Hesselholt \cite{GHbi} and Morrow \cite{Morrowpro2}, and the $p$-adic Beilinson fiber square construction of Antieau--Mathew--Morrow--Nikolaus \cite{NikolausBFSq}. 

Our methods apply to more general coefficient rings. We show:

\begin{theorem}\label{thm:MainIntro}
Suppose $G$ is a finite group, $n\geq 0$, and $R$ is a commutative ring
which is Noetherian, regular, and flat over $\BZ$. Suppose that for all primes $p$
dividing $|G|$  the quotient $R/p$ is regular.   Then there exists a positive integer  which  annihilates $\RNK_n(R[G])$ and depends only on $n$, $G$, and the set of primes  that  are not units in $R$. 
\end{theorem}

In fact,   the annihilator in the above theorem  is very explicit when the order $|G|$ is \emph{not} divisible by a ``small" prime, where small here means less than $3+n/2$.
See \S \ref{ss:results} and in particular Theorem \ref{thm:general}. 

\subsection{}\label{ss:introtwisted}  We can also treat general Farrell Nil groups:
  \begin{theorem}\label{FarrellNilThm_Intro}
Let $G$ be a finite group,  $\alpha: G\to G$ an automorphism, and $R$ a commutative ring which  is regular, Noetherian, flat over $\BZ$, and such that $R/p$ is regular for all primes $p$ that divide $|G|$. Then, for any $n\in\BZ$, the group $\NK_n(R[G],\alpha)$ is annihilated by a power of $|G|$.
\end{theorem}

By \cite[Theorem C]{LPW}, this implies that, for every $n\in\BZ$ and every automorphism $\alpha:G\to G$, there is a finite abelian group $H$ such that  $\RNK_{n}(\BZ[G], \alpha)\simeq \oplus_\infty H$. Combining this with work of Davis--Khan--Ranicki \cite{DKR}, we obtain
\begin{theorem}\label{cor2_Intro}
For every $n\in \BZ$ and every virtually   cyclic group $\Gamma$, the group $\RK_n(\BZ[\Gamma])$ is isomorphic to a direct sum 
$
\Lambda\oplus (\oplus_\infty T)
$
of a finitely generated abelian group $\Lambda$
with an infinite countable direct sum  of copies of a finite abelian group $T$.
\end{theorem}

Thus, even though the $\RK$-groups of integral group rings of  virtually cyclic groups need
not be finitely generated, all of their ``infinite behavior" is accounted for by countably
many copies of a single finite abelian group.

These  results (see Theorems \ref{thm1_appB} and \ref{cor2_appB} in the text) generalize Theorems \ref{ZresultIntro}, \ref{thm:Countable} and \ref{thm:MainIntro} above. In the genuinely twisted case, i.e. when $\alpha:G\to G$ is not the identity, our method of proof of Theorem \ref{FarrellNilThm_Intro} differs from the one we use for the untwisted case: It is based on an idea described in the Appendix by Morrow, and it does not seem to lead easily to specific bounds for the exponent of $\NK_n(R[G],\alpha)$. Giving such  exponent bounds in general remains an interesting problem.


\subsection{}   Let us  sketch the proof of Theorem \ref{thm:MainIntro}.

We first reduce to the case of a $p$-group $G$ and then, by some  quite general induction arguments, we further reduce to handling the case that $G$ is a cyclic $p$-group. These reductions crucially use the results of Hambleton-L\"uck \cite{HLuck} and of Weibel \cite{WMV}.  

We now assume $G$ is a cyclic $p$-group. Suppose $M_G$ is the maximal order in $\Q[G]$ and set, for simplicity, $\A=R[G][t]$ and $\M=(M_G\otimes_\BZ R)[t]$.
Under our assumptions, $M_G\otimes_\BZ R$ is regular and so $\RK_n(\M)=\RK_n(M_G\otimes_\BZ R)$ by homotopy. We first observe that 
  $\RNK_n(R[G])$ is a subgroup of the kernel  of the natural map $\RK_n(\A)\to \RK_n(\M)$, by the above and the Bass-Heller-Swan formula.
We then use a  homotopy Cartesian diagram of  spectra
 \begin{equation}\label{CDintro}
\begin{aligned}
\xymatrix{
\RK(\A) \ar[r]\ar[d] & \RK(\M) \ar[d]\\
\RK^c(\wh \A)\ar[r] &\RK^c(\wh \M)\\
}
\end{aligned}
\end{equation}
which is obtained in work of Morrow \cite{Morrowpro}, \cite{Morrowpro2}, by using pro-excision.   
Here, we have ($p$-adically) continuous $\RK$-theory
\[
\RK^c(\wh \A):=\prolim_a \RK(\A/p^a),\qquad  \RK^c(\wh \M):=\prolim_a \RK( \M/p^a),
\]
where $\wh \A=\prolim_a  \A/p^a$ and $\wh \M=\prolim_a   \M/p^a$ are the $p$-adic completions. 

The diagram \eqref{CDintro} implies a homotopy equivalence 
between the fibers of $\RK(\A) \to  \RK(\M)$ and $\RK^c(\wh \A) \to  \RK^c(\wh\M)$.  
This can be used to bound the kernel of $\RK_n(\A)\to \RK_n(\M)$, and, by the above, also $\RNK_n(R[G])$, in terms of $\RK^c(\wh \A) \to  \RK^c(\wh\M)$.
Indeed, if $\RF$ is the common homotopy fiber, we have an exact sequence
\[
\pi_n(\RF)\to \RK_n(\A)\to \RK_n(\M).
\]
 Hence, it is enough to find an exponent  for $\pi_n(\RF)$.
 
 We find  this exponent in two steps by considering localization sequences for continuous $\RK$-theory with respect to the ideal $(p)$ in the $p$-adically complete algebras $\wh \A$ and $\wh \M$.  Localization introduces  a subgroup of the group $\pi_n(\RF)$ given via continuous $\RK$-groups   relative to $(p)$. We find an integer annihilating this subgroup in \S \ref{s:Beilinson} after comparing to a corresponding relative cyclic homology group via the $p$-adic Beilinson isomorphism of \cite{NikolausBFSq}. On the other hand, the quotient relates to $\RK$-groups of algebras in characteristic $p$.  We find an integer annihilating this quotient in \S \ref{s:truncated} by using results of Hesselholt--Madsen. 
 
\subsection{}   In response to a first version of the paper,  M. Morrow pointed out that the weaker result that $\RNK_n(R[G])$ has (unspecified) finite exponent when $G$ is a cyclic $p$-group can be shown more directly, by just combining pro-excision with the main result of \cite{GH2}. This shorter argument is explained in the Appendix, see the proof of Theorem \ref{thm1_app}.
It avoids the use of the $p$-adic Beilinson isomorphism but, in contrast to our approach, does not seem to lead to explicit bounds for the exponent. In fact, provided we are not aiming for specific exponent bounds, Theorem \ref{thm2_app}  provides a more general result: For each $n$, the Nil group $\RNK_n(X)$ of an excellent Noetherian scheme $X$, with $X[1/p]$ regular,
is annihilated by an unspecified finite power of $p$, provided $X$ admits a suitable resolution of singularities.

Our results in the twisted case described in \S \ref{ss:introtwisted} use the method in the proof of Theorem \ref{thm1_app}, extended to certain noncommutative rings.

\subsection{}   The paper is organized as follows. 

In \S \ref{s:1} we first give some preliminaries.  
 We then state our main results in the untwisted case $\alpha={\rm id}$, in \S \ref{ss:results}. These are stated for   general coefficient rings and we take some care 
 to give, as much as is feasible, explicit bounds for the exponents.

In \S \ref{s:reduction} we reduce the proofs to showing a single statement, Theorem \ref{thmexpo},  that gives an exponent when the group $G$ is a cyclic $p$-group. Sections 4--6 are devoted to proving this statement. 

In \S \ref{sec:ProMV},
we give the arguments regarding pro-excision, roughly following the outline above. In fact, we find that it is more effective to work throughout directly with nil $\RK$-theory, i.e. relative to the ideal $(t)$. We somewhat modify the argument in the sketch above and use instead a fibration sequence
\[
\RNK(\BZ[G])\to \RNK^c(\BZ_p[G])\to \RNK^c(M_G\otimes_\BZ\BZ_p),
\]
obtained from \eqref{CDintro}, see \eqref{nfib1}. We then apply localization relative to the ideal $(p)$.
With the resulting exact sequences, 
\eqref{pinFexact} and  \eqref{pinF/p}, we set the stage for obtaining the annihilators needed for the proof of Theorem \ref{thmexpo}, see \S \ref{ss:strategy} and the preceding discussion.

Finally, these annihilators are obtained in  Sections \ref{s:truncated}  and \ref{s:Beilinson}, as was also explained above. At the end of \S \ref{s:Beilinson},
we reexamine the case $n=1$ and show how a somewhat different argument leads to a lower exponent bound (see \S \ref{ss:n1case}).

In \S \ref{s:AppendixB}, we show the boundedness result  for the Farrell  groups $\NK_n(R[G],\alpha)$, Theorem \ref{FarrellNilThm_Intro}. We then deduce Theorem \ref{cor2_Intro}. The results of this section were obtained after the Appendix was written and the proofs use a similar method. In fact, 
the reader is advised to examine the proof of Theorem \ref{thm1_app} before proceeding to \S \ref{s:AppendixB}.

The paper concludes with the Appendix by M.~Morrow.

\smallskip 

\noindent {\bf Acknowledgements:} We would like to thank L. Hesselholt, A.~Mathew,  M.~Morrow, N.~Riggenbach and C.~Weibel for very useful email exchanges. 
We are especially grateful   to  C. Weibel, who suggested to us the use of pro-excision,  and to M.~Morrow for writing the Appendix, which also inspired the results in \S  \ref{s:AppendixB}.

 \section{Preliminaries and statement of results}\label{s:1}

\subsection{ Definitions and preliminaries}
Let $ R$ be a commutative ring, let $G$ be a finite group and $n$ a non-negative integer.

The following two results are central: see \cite[5.3.30]{Rosenberg}.

\begin{theorem}\label{thm:BHS}(Bass-Heller-Swan)
Let {\bf Nil}($S$) be the category whose objects are pairs $(P,A)$ consisting of a finitely generated projective $S$-module P and a nilpotent $S$-endomorphism $A$ of $P$.  The $\RK$-groups of ${\mathrm {Nil}}(S)$ split naturally as $\RK_i(S) \oplus \mathrm{Nil}_i(S)$.     
Define 
\[
\RNK_n(S) = \mathrm{Nil}_{n-1}(S).
\]
There are natural isomorphisms
\begin{equation*}
\RK_{n}( R[ t] [ G] )  \xrightarrow{\sim} \RNK_{n}( R [ G] ) \oplus \RK_{n}( R[ G]),
\end{equation*}
\begin{equation*}
\RK_{n}( R[ t,t^{-1}] [ G] )  \xrightarrow{\sim} \RNK_{n}( R [ G] ) \oplus \RK_{n}( R[ G]
) \oplus \RNK_{n}( R [ G] ) \oplus
\RK_{n-1}( R[ G] ) .
\end{equation*}
\end{theorem}

\begin{theorem}
\label{cor:easyregular}
If $R[G]$ is a (left) regular ring, then $\RNK_{n}(R[G]
) =(0) $.
\end{theorem}

Note that by the above, the Nil  groups are  also relative $\RK$-groups (\cite{WeibelKbook})
\begin{equation}\label{nilrelative}
\RNK_n(R[G])=\RK_n(R[G][t], (t)).
\end{equation}

We will also use the following, shown by Weibel \cite{WMV}, after work of Vorst \cite{Vorst} and Stienstra \cite{Stienstra}.

\begin{theorem}\label{thm:localization} (cf. \cite[Cor. 6.4]{WMV}) Suppose that $R$ is flat over $\BZ$ and $S$ is a multiplicative subset of the integers $\BZ$. Then, there
are isomorphisms
\begin{equation}\label{eq:localization}
S^{-1}\BZ\otimes_\BZ \RNK_n(R[G])\xrightarrow{\sim} \RNK_n((S^{-1} R)[G]),
\end{equation}
where $S^{-1}\BZ$ and $S^{-1}R=S^{-1}\BZ\otimes_\BZ R$ denote the localizations. \qed
\end{theorem}

 \begin{corollary}\label{cor:localization} 
 Suppose $R$ is regular and flat over $\BZ$. Then, for all $n\geq 0$,   
 \[
\BZ[ |G|^{-1}]\otimes_\BZ \RNK_n(R[G])=(0).
 \]
 \end{corollary}
 \begin{proof}
 This follows from Theorem \ref{thm:localization} and Theorem  \ref{cor:easyregular} since when $R$ is regular, the localization
$\BZ[ |G|^{-1}] \otimes_\BZ R[G]$ is left regular.
\end{proof}
 
 \subsection{Results in the untwisted case}\label{ss:results}
 
 We always assume $n\geq 0$ and $G$ is a finite group. Let  $R$ be a commutative ring which satisfies the assumptions of Theorem \ref{thm:MainIntro}:
\smallskip

\begin{quote} $R$ is Noetherian, regular, and flat over $\BZ$ and 
 $R/p$ is regular for all primes $p$
dividing the order $|G|$.
\end{quote}
\smallskip

Our main results are given in the next paragraphs.  
These produce a very explicit bound for the exponent of $\RNK_n(R[G])$ when the order of $G$ is not divisible by any prime $p<3+n/2$.

\subsubsection{General coefficients}

 We denote by $e(G)$ the exponent of $G$. 
 For each prime $p$ which divides $|G|$ we denote by $G_p$ a $p$-Sylow subgroup of $G$. 

\begin{theorem}\label{thm:general}  
Assume that $R$ satisfies the assumptions of Theorem \ref{thm:MainIntro}. 
There exist integers $a(p, n)\geq 0$, which are equal to $0$ if $p\geq 3+n/2$ and which are independent of 
$R$ and $G$, so that  $\RNK_{n}(R[G])$ is annihilated by 
\[
 \prod_{p||G|, p\not\in R^*} e(G_p)^{n+3}|G_p|\, p^{ a(p, n) +3},
\]
  the product being  over the prime divisors of $|
G|$ which are not units in $R$.
\end{theorem}

The integers $a(p, n)$ are defined in \S \ref{ss:Finale}.
This obviously implies Theorem \ref{thm:MainIntro} of the introduction.
We obtain a  sharper result when $G$ is a cyclic $p$-group with $p\geq 3+n/2$.

\begin{theorem}\label{pcyclicIntro}
Assume that $R$ is as above and that $G$ is a non-trivial cyclic group of order $p^m$,  with $p\geq 3+n/2$.
Then $\RNK_{n}(R[G])$ is annihilated by
$
p^{(n+3)m+4 } .
$
\end{theorem}

\begin{Remark}
{\rm It is clear that the exponents provided by these theorems are not always optimal. This can be seen   by comparing with the previous results.
 For example, $\RNK_1(\BZ[C_p])=0$ by \cite[Cor. 1.1]{HLuck} or \cite{BassMurthy}, while   the above result only gives $p^8\cdot
\RNK_1(\BZ[C_p])=0$ when $p\geq 5$. However, see \S \ref{ss:n1case} for an improvement.}
\end{Remark}

 \subsubsection{Rings of integers}
 
The above  immediately apply to $R=O_L$, the ring of integers in a number field $L$,
when $L$ is unramified over $\mathbb{Q}$  at all prime divisors of the group order 
$|G|$.  In particular, they apply to $L=\bbQ$ and we obtain Theorem \ref{ZresultIntro} in the introduction. Theorem \ref{thm:Countable} 
then follows by using \cite[Theorem C]{LPW}.

\section{Reduction to cyclic $p$-groups}\label{s:reduction}

In this section we will show how to reduce the proofs to considering the case where $G$ is a cyclic $p$-group. 
More precisely, we show how Theorems   \ref{thm:general}  and \ref{pcyclicIntro}  follow from  Theorem \ref{thmexpo} below,  which pertains to cyclic $p$-groups. We accomplish this in two steps: first we reduce to the case where $G$ is a $p$-group; then we introduce the notions of Frobenius functors and Frobenius modules; this then permits us to use induction techniques to reduce to considering cyclic $p$-groups. 

\subsection{Reduction to $p$-groups}

Let $p$ be a prime and denote by $M_{(p)}$ the localization of an abelian group $M$ at $\mathbb{Z} - (p)$. Fix  $p$ and 
let $\mathcal{P}( G) $ denote the set of $p$-subgroups of $G$.
For an arbitrary ring $R $ and for each $P\in \mathcal{P}( G) $, Hambleton and L\"uck define in \cite{HLuck} a homomorphism
\begin{equation*}
\Phi _{P}:\RNK_{n}( R[ P] ) \rightarrow \RNK_{n}( R%
[ G] ). 
\end{equation*}
They show:

\begin{theorem}\label{HambletonL}  (\cite[Theorem A]{HLuck})
 The homomorphism
\begin{equation*}
\Phi =\sum_P \Phi _{P}:\bigoplus _{P\in \mathcal{P}( G)
}\RNK_{n}( R[ P] ) _{( p) }\rightarrow
\RNK_{n}( R[ G] ) _{( p) }
\end{equation*}
is surjective.
\end{theorem}

\begin{Remark}
{\rm The maps $\Phi _{P}$ are not given simply as induction maps from $p$-subgroups: they are more complicated and use elements of $G$ of order prime to 
$p$ and the Verschiebung homomorphisms. }
\end{Remark}

Assume now $R$ is regular and flat over $\BZ$.  

By Corollary \ref{cor:localization},   $\RNK_n(R[P])[1/p]=0$, hence
\[
\RNK_{n}( R[ P] )\xrightarrow{\sim}  \RNK_{n}( R [ P] )_{(p)}.
\]
By Theorem  \ref{thm:localization}, there is an isomorphism
\[
\RNK_{n}( R[ P] )_{(p)}\xrightarrow{\sim}   \RNK_{n}( R_{(p)}[ P] ).
\] 
Combining the two, we obtain an isomorphism
\begin{equation}\label{localizationiso}
\RNK_{n}( R[P] ) \xrightarrow{\sim}  \RNK_{n}( R_{(p)}[ P] ).
\end{equation}
 
 \subsection{Frobenius structure}
 In this subsection, $R$ is an arbitrary commutative $A$-algebra, where $A$ is either $\BZ$ or $\BZ_{(p)}$, and $H$ is a finite group. From \cite{CR94}, we have the Grothendieck ring of finitely generated $A[H]$-modules which are projective as $A$-modules, which is denoted by $\RG^{A}_{0}( A[ H])$. 
Since $A=\mathbb{Z}$ or $\BZ_{(p)}$ is regular, from \cite[p. 22]{CR94} we know that the Grothendieck group  $\RG_{0}(A[H] )$ of finitely generated $A[H]$-modules is naturally isomorphic to the ring  $\RG_{0}^{A}(A[ H] )$.

 We recall the notions of a Mackey functor and a Green ring acting on a Mackey
functor, see  e.g. \cite[p. 246]{Ol88}.
From   \cite[p. 200]{HLuck}  the functor $H\mapsto \RNK_{n}( R[ H]
) $ is a Mackey functor on the subgroups $H$ of $G$.  Furthermore, as in \emph{loc. cit.} page 210,
  the Green ring $\RG^A_0(A[G])$ acts on the Mackey functor $\RNK_{n}(R[H])$; that is to say that the 
 functor $\RNK_{n}(R[H])$ is a Green module for the Green ring  $\RG^A_0(A[G])$\footnote{In   \emph{loc. cit.} this is only stated for $A=\BZ$, but the case $A=\BZ_{(p)}$ is similar.}. Thus, in the terminology of Curtis--Reiner  \cite[\S 38.1, \S 38.7]{CR94}, 
we know that $H\mapsto \RNK_{n}(R[H])$ and $
H\mapsto  \RG^A_0(A[H])$ are Frobenius functors; and that $H\mapsto \RNK_{n}( R[ H] ) $ is
a Frobenius module for the Frobenius functor $H\mapsto \RG^A_0(A[H])$.
(See also  \cite[Cor. 2]{Harmon} for the case $n=1$.)

\subsection{$p$-groups and cyclic groups.}

  In  this subsection, initially we reduce to the case that $G$ is a $p$-group and subsequently a cyclic $p$-group.

\subsubsection{Artin induction}

Recall that the Artin exponent $A(G)$ of 
$G$ is defined (\cite{Lam}) to be the exponent of the quotient group 
\begin{equation*}
\frac{\RG_{0}^\Q( \mathbb{Q}[ G] ) }{\sum_{C}
\mathrm{Ind}_{C}^{G}\,\RG^\Q_{0}( \mathbb{Q}[ C] ) }
\end{equation*} 
where the sum is over cyclic subgroups $C$ of $G$.
In fact, this quotient is   a commutative ring, since the denominator is an ideal of $ 
\RG^\Q_{0}( \mathbb{Q}[ G] )$. 

From Lam's theorem (\cite{Lam}, \cite[Theorem 76.15]{CR94})  we know that
if $G$ is a non-cyclic $p$-group of order $p^{m}$ then 
\begin{equation}\label{ArtinExpValue}
A( G)\, |\, p^{m-1},
\end{equation}
with equality when $p$ is odd.

By Swan's Theorem \cite[Thm. 39.10]{CR94} the natural map
\[
\RG^{\BZ_{(p)}}_{0}( \mathbb{Z}_{(p)}[ G] ) \to \RG^\Q_{0}( \Q[ G] )
\]
is an isomorphism. Hence, by \cite[p. 785, Cor. 76.8]{CR94}, in $\RG^{\mathbb{Z}_{(p)}}_{0}( \mathbb{Z}_{(p)}[ G] )$  we can write 
\begin{equation}\label{Indformula}
A(G) =\sum_{C}\; \mathrm{Ind}_{C}^{G}(\kappa _{C})
\end{equation}
with   $\kappa _{C}\in \RG^{\mathbb{Z}_{(p)}}_{0}(\mathbb{Z}_{(p)}
[ C] )$.   We now obtain
 \begin{theorem}\label{reduction to cyclic}
 Suppose that $R$ is a $\BZ_{(p)}$-algebra.
For $n\geq 0$, if $M\geq 1$ annihilates $\RNK_{n}( R[ C] ) $ for all cyclic
subgroups $C$ of $G,$  then $A( G) M$ annihilates $\RNK_{n}( R[ G
] )$. 
\end{theorem}

\begin{proof} Suppose $\kappa \in \RNK_{n}( R[ G] ) $.  With the notation of (\ref{Indformula})  we have 
\begin{eqnarray*}
A( G) M\cdot \kappa = M\cdot\sum_{C}   ( \mathrm{Ind} 
_{C}^{G}\kappa _{C}) \cdot \kappa  
=\sum_{C}  \mathrm{Ind}_{C}^{G}( M\cdot \kappa _{C}\cdot \text{ 
\textrm{Res}}_{C}^{G}\kappa) =0, 
\end{eqnarray*}
 by using the Frobenius reciprocity relation for Frobenius modules. \end{proof}
 
 \begin{corollary}\label{killNKR} Assume that $R$ is regular and flat over $\BZ$.
Let $G$ be a non-cyclic $p$-group of order $p^{m}$ and
suppose $n\geq 0$ and that $p^{e}$ annihilates $\RNK_{n}( R[ C] ) $
 for all cyclic subgroups $C$ of $G$. Then  $p^{e+m-1}$
annihilates $\RNK_{n}( R[ G] )$. 
 \end{corollary}

\begin{proof}
By \eqref{localizationiso} applied to $P=G$ and all its subgroups, we can replace $R$ by its localization $R_{(p)}$ and so we can assume that $R$ is a $\BZ_{(p)}$-algebra.
The result then follows from Theorem \ref{reduction to cyclic} and \eqref{ArtinExpValue}.
\end{proof}

\subsubsection{Reduction to the cyclic $p$-group case}

This now reduces us to considering   cyclic $p$-groups. The bulk of the rest of the paper 
concerns the proof of the following result:

  \begin{theorem}\label{thmexpo}   For each prime $p$ and $n\geq 0$, there is an integer $a(p, n)\geq 0$ 
which is equal to $0$ when $p\geq 3+n/2$ and is such that the following holds. 
For any  cyclic group $G$ of order $p^m$ and ring $R$ satisfying the assumptions of Theorem \ref{thm:MainIntro} with $p\not\in R^*$, we have 
\[
p^{(n+3)m+4+a(p, n) }\cdot \RNK_n(R[G])=(0).
\]
  \end{theorem}

  Note that Theorem \ref{thmexpo} immediately implies 
  Theorem \ref{pcyclicIntro}.

  \subsubsection{Proof of Theorem \ref{thm:general},  assuming Theorem \ref{thmexpo}.}
  First we see that, by Theorem \ref{HambletonL} and Corollary \ref{cor:localization}, it is enough to assume that 
  $G$ is a $p$-group, i.e. $G=G_p$, for $p$ not a unit in $R$. 
   For $G_p$ cyclic, Theorem \ref{thmexpo}
  already gives a stronger result. Suppose now $G_p$ is not cyclic.
 Since by Theorem \ref{thmexpo},  $e(G_p)^{n+3}p^{4+a(p,n) }$ annihilates $\RNK_n(R[C])$ for all cyclic subgroups $C\subset G_p$,
  we can apply Corollary \ref{killNKR} with $e=(n+3)\log_p(e(G_p))+4+a(p,n)$.  This gives the result. \qed

Note that Theorem \ref{thmexpo} is stronger, i.e. it
gives a sharper (smaller) exponent bound than  Theorem \ref{thm:general} specialized to cyclic $p$-groups.

 \section{Pro-excision, Mayer-Vietoris and localization}\label{sec:ProMV}
 
 \subsection{Preliminaries}
 
In all of \S \ref{sec:ProMV} we assume that $G$ is a {\sl cyclic $p$-group} and that $R$ satisfies the assumptions of Theorem \ref{thm:MainIntro}. In this section, we reduce the proof of Theorem \ref{thmexpo} to 
producing two exponent bounds (A) and (B) as in \S \ref{ss:strategy}. 

We let $M_G$ be the maximal order in $\Q[G]$ containing the integral group ring $\BZ[G]$.

 \begin{lemma}\label{lemma:conductor}
Assume $|G|=p^m$. Then
\[
p^m M_G\subset  \BZ[G]\subset M_G,
\]
so $I=p^mM_G$ is a  common ideal of $\BZ[G]$ and $M_G$. We then also have
\[
 p^{2m}\BZ[G]\subset I^2\subset p^m\BZ[G]\subset I,
\]
as ideals of $\BZ[G]$.
\end{lemma}

\begin{proof}
Recall that $M_G$ is the normalization of $\BZ[G]$ in $\Q[G]$. 
 The first claim follows from \cite[Prop. 27.1]{CR90}. 
For the second, observe that 
\[
I^2=p^{2m}M_G\cdot M_G\subset p^{2m}M_G=p^m I\subset p^m\BZ[G],
\]
and $p^m\BZ[G]\subset p^mM_G=I\subset \BZ[G]$.
\end{proof}

\begin{lemma}\label{regular}
The ring $M_G\otimes_\BZ R$ is regular.
\end{lemma}

\begin{proof}  This follows from the proof of Proposition \ref{propB}. However, in the cyclic case the argument is a lot simpler: The ring $  M_G\otimes_\BZ R$ is the direct product of rings of the form $R[\zeta_{p^r}]:=R\otimes_\BZ \BZ[\zeta_{p^r}]$, with $\BZ[\zeta_{p^r}]$ the $p^r$-th cyclotomic ring of integers. Let $\varpi_r$ be a uniformizer of $\BZ_p[\zeta_{p^r}]$. Then 
$
R[\zeta_{p^r}]/(\varpi_r)=R/p,
$
which is assumed to be regular. Since $R$ is $\BZ$-flat,   $R[\zeta_{p^r}]$ is a flat $\BZ[\zeta_{p^r}]$-algebra. Since $R[\zeta_{p^r}]/(\varpi_r)$ is regular, $R[1/p]$ is regular, and $\BZ[1/p]\to \BZ[\zeta_{p^r}][1/p]$ is \'etale, we conclude that $R[\zeta_{p^r}]$ is regular. 
 \end{proof}

\subsection{Pro-excision and Mayer-Vietoris} 

In what follows, for simplicity, we set
\[
\A=R[t][G]=\BZ[G]\otimes_\BZ R[t], \quad \M= M_G\otimes_\BZ R[t]. 
\]
We also set 
\[
J=I\otimes_\BZ R[t]=p^m M_G\otimes_\BZ R[t]\subset \A\subset \M,
\]
which is a common ideal of $\A$ and $\M$.
 We will apply pro-excision to the pairs $(\A, J)$ and $(\M, J)$.
 Indeed, by \cite[Thm. 0.2, Thm. 0.3, Cor. 0.4]{Morrowpro2} (which builds on results of \cite{GHbi}) 
 pro-excision for $J$ holds. As in Remark 1.5 of \cite{Morrowpro}, we see that
 this implies that the homotopy limit
 \[
 \prolim_a \RK(\A, \M, J^a)
 \]
 of bi-relative $\RK$-theory spectra is contractible. This gives, as in \emph{loc. cit.},  
a
homotopy cartesian diagram of (connective) spectra 
 \begin{equation}\label{CD}
\begin{aligned}
\xymatrix{
\RK(\A) \ar[r]\ar[d] & \RK(\M) \ar[d]\\
\RK^c(\wh \A)\ar[r] &\RK^c(\wh \M)\\
}
\end{aligned}
\end{equation}
with
\[
\RK^c(\wh \A):=\prolim_a \RK( \A/J^a),\qquad  \RK^c(\wh \M):=\prolim_a \RK(\M/J^a).
\]
Here, $\wh \A=\prolim_a \A/J^a\simeq \prolim_a \A/p^a$ and $\wh \M=\prolim_a \M/J^a\simeq \prolim_a \M/p^a$,
i.e. they are the $p$-adic completions of $A$ and $\M$ respectively, by cofinality using the inclusion
 in   Lemma \ref{lemma:conductor}.
Also, similarly
\[
\RK^c(\wh \A) \simeq \prolim_a \RK( \A/p^a), \quad
 \RK^c(\wh \M) \simeq \prolim_a \RK(\M/p^a).
\]
Set $\RF=\RK(\A, \M)$ for the fiber of $\RK(\A)\to \RK(\M)$. Since \eqref{CD} is homotopy Cartesian, 
the two horizontal maps in \eqref{CD} have homotopically equivalent fibers. Hence, there is also a fiber sequence  
\begin{equation}\label{fib1}
\RF\to \RK^c(\wh \A)\to \RK^c(\wh \M).
\end{equation}
We note from \eqref{CD} that we also obtain an exact Mayer-Vietoris sequence of homotopy groups
\begin{equation}\label{MV}
 \RK_{n+1}(\M)\oplus \RK^c_{n+1}(\wh \A)\to \RK^c_{n+1}(\wh \M)\to \RK_n(\A) \to \RK_{n}(\M)\oplus \RK^c_{n}(\wh \A)\to \RK^c_{n}(\wh \M)\ .
\end{equation}
Set
$
\K_n:=\ker( \RK_n(\A) \to \RK_{n}(\M)).
$
Then, \eqref{MV} gives  an exact sequence
\begin{equation}\label{MV2}
{\rm {coker}}(\RK^c_{n+1}(\wh \A)\to \RK^c_{n+1}(\wh \M))\to \K_n\to \ker(\RK^c_{n}(\wh \A)\to \RK^c_{n}(\wh \M))\ .
\end{equation}
Since $ M_G\otimes_\BZ R$ is regular,
$
\RK_n(\M)=\RK_n(( M_G\otimes_\BZ R)[t])=\RK_n( M_G\otimes_\BZ R).
$
So $\K_n$ is also the kernel of the composition
\[
\RK_n(\A)\xrightarrow{t\mapsto 0}  \RK_n(R[G])\to \RK_n(M_G\otimes_\BZ R).
\]
We conclude that
$\RNK_n(R[G])$ is a subgroup of $\K_n$.

\begin{Remark}
{\rm  The sequences \eqref{MV} and \eqref{MV2} are used
only in \S \ref{ss:n1case} and not in the argument of the proof of Theorem \ref{thmexpo}. For that, we argue directly with Nil $\RK$-groups as follows.}
\end{Remark}

 \subsection{A version for Nil $\RK$-theory} For a ring $B$, we set $\RNK(B):=\RK(B[t], (t))$ and 
\[
\RNK^c(\wh B):= \prolim_a \RNK( B/ p ^a):=\prolim_a \RK( B[t]/p^a, (t)).
\] 
We have natural splittings
\[
\RK(B[t])=\RNK(B)\oplus \RK(B),\quad \RK^c(\wh B\lan t\ran)=\RNK^c(\wh B)\oplus \RK^c(\wh B),
\]
where $\wh B=\prolim_a B/p^a$, $ \wh B\lan t\ran=\prolim_a B[t]/p^a$. 

The fiber $\RF={\rm Fib}(\RK(A)\to \RK(\M))$ decomposes 
\[
\RF=\RNF\oplus {\rm Fib}(\RK(R[G])\to \RK(M_G\otimes_\BZ R)),
\]
where $\RNF$ is the fiber of $\RNK(R[G])\to \RNK(M_G\otimes_\BZ R)$. However, by Lemma \ref{regular} and  the homotopy theorem (Theorem \ref{cor:easyregular}), $\RNK(M_G\otimes_\BZ R)$ is contractible. Hence, $\RNK(R[G]) \simeq \RNF$. Since the diagram \eqref{CD} respects these splittings we also have a fibration
\begin{equation}\label{nfib1}
\RNF\to   \RNK^c(\wh R[G])\to \RNK^c(M_G\otimes_\BZ \wh R).
\end{equation}

\subsection{Localization}
We now consider $p$-adically continuous $\RK$-theory relative to the ideal $(p)$ in $\A$ and in $\M$. We set
\[
\RK^c(\wh \A, (p)):=\prolim_a \RK(\A/p^a, (p)),\quad \RK^c(\wh \M, (p)):=\prolim_a \RK(\M/p^a, (p)).
\]
Recall that the fiber $\RF$ above is equivalent to the fiber of $\RK^c(\wh \A)\to \RK^c(\wh \M)$. We have a diagram of fibrations
 \begin{equation}\label{CDfib1}
\begin{aligned}
\xymatrix{
\RF(-,(p))\ar[r]\ar[d] & \RF\ar[r]\ar[d] & \RF(/p)\ar[d]\\
\RK^c(\wh \A, (p))\ar[r]\ar[d] &\RK^c(\wh \A) \ar[r]\ar[d] & \RK( \A/p)\ar[d]\\
\RK^c(\wh \M, (p)) \ar[r] & \RK^c(\wh \M)\ar[r] & \RK( \M/p)\\
}
\end{aligned}
\end{equation}
which defines $\RF(-, (p))$ and $\RF(/p)$. 

Similarly, we set
\[
\RNK^c(\wh R[G], (p)):=\prolim_a \RNK(R[G]/p^a, (p)),
\]
\[
  \RNK^c(M_G\otimes_\BZ \wh R, (p)):=\prolim_a \RNK(M_G\otimes_\BZ  R/p^a, (p)).
\]
We also have a diagram of fibrations
\begin{equation}\label{CDfib2}
\begin{aligned}
\xymatrix{
\RNF(-,(p))\ar[r]\ar[d] & \RNF\ar[r]\ar[d] & \RNF(/p)\ar[d]\\
\RNK^c(\wh R[G], (p))\ar[r]\ar[d] &\RNK^c(\wh R[G]) \ar[r]\ar[d] & \RNK( R[G]/p)\ar[d]\\
\RNK^c(M_G\otimes_\BZ \wh R, (p)) \ar[r] & \RNK^c(M_G\otimes_\BZ \wh R)\ar[r] &\RNK( M_G\otimes_\BZ R/p)\\
}
\end{aligned}
\end{equation}
which  splits off as a direct summand of \eqref{CDfib1}.

Using $\RNK(R[G])\simeq \RNF$, and the resulting exact sequence
\begin{equation}\label{pinFexact}
  \pi_n(\RNF(-, (p)))\to  \RNK_n(R[G])\simeq \pi_n(\RNF)\to \pi_n(\RNF(/p))  ,
\end{equation}
we see that it is enough to bound the exponents of 
$\pi_n(\RNF(-, (p)))$ (or of its supergroup $\pi_n(\RF(-,(p)))$) and of $\pi_n(\RNF(/p))$.
Note that by the above, 
we have an exact sequence
\begin{equation}\label{pinF/p}
 \RNK_{n+1}(M_G\otimes_\BZ R/p) \to  \pi_n(\RNF(/p))\to \RNK_n(R[G]/p).
\end{equation}

 \subsection{The strategy for the proof of Theorem \ref{thmexpo} for cyclic $p$-groups}\label{ss:strategy}
   Our    goal is to find an exponent that annihilates:
  \begin{itemize}
  \item[A)]  $\pi_n(\RF(-, (p)))$, by comparing with relative
  cyclic homology,  
  
  \item[B)]  $\pi_n(\RNF(/p))$, by using \eqref{pinF/p} and giving exponents for the groups $\RNK_n(R[G]/p)$ and $\RNK_{n+1}(M_G\otimes_\BZ R/p)$. 
    \end{itemize}
  
  Combining these with \eqref{pinFexact}  gives an annihilator of  $\RNK_n(R[G])$ when $G$ is a cyclic $p$-group  and thereby leads to a proof of Theorem \ref{thmexpo} (see \S \ref{ss:Finale}).

    In the next section we discuss (B).

\section{$\RK$-theory of some truncated polynomial ${\mathbb F}_p$-algebras}\label{s:truncated}

We continue to assume that $G$ is a cyclic $p$-group and that $R$ satisfies the assumptions of Theorem \ref{thm:MainIntro}. We will use results of Hesselholt-Madsen \cite{HesseMad} and Hesselholt \cite{HesseHandbook} to find  exponents for the groups 
$\RNK_{n+1}(M_G\otimes_\BZ R/p)$ and $\RNK_n(R[G]/p)$.

\subsection{Nil $\RK$-groups of some ${\mathbb F}_p$-algebras}
The rings $M_G\otimes_\BZ R/p$ and $R[G]/p$ are both direct products of ${\mathbb F}_p$-algebras of the form $B[x]/(x^e)$ in which $B=R/p$ is regular Noetherian  and  $e\geq 1$.   

We will denote, as usual, by $\lfloor \alpha \rfloor$ the \emph{floor} of the real number $\alpha$, i.e. the greatest integer less than or equal to $\alpha$. 
For $x\geq 0$, we define 
\[
\log^*_p(x)=\begin{cases} \log_p(x), & \hbox{\rm if}\ x>0,\\ -1, &\hbox{\rm if}\ x=0,
\end{cases}
\]
with $\log_p(x)$ the real logarithm of $x$ with base $p$.
 
\begin{proposition}\label{lem:dumbbound}  Suppose $i \ge 1$, $B$ is a regular noetherian ring of characteristic $p$  and $e = p^\nu e'$ with $e'$ prime to $p$.  Then the group $\RNK_{i-1}(B[x]/(x^e) )$ is annihilated by $p^z$ when 
\begin{equation}
\label{eq:zdef}
z = \lfloor \log^*_p(e\lfloor i/2\rfloor ) \rfloor  + \lfloor \log^*_p(\lfloor (i-1)/2\rfloor) \rfloor + 2.
\end{equation}  
\end{proposition}

\begin{proof} In \cite[Corollary 14]{HesseHandbook} one finds the following exact sequence involving the  nil $\RK$-group $\RNK_{i-1}( B[x]/(x^e) )$ when $e = p^{\nu} e'$ with $e'$ prime to $p$:

  Let $I_p$ be the set of positive integers that are prime to $p$.  
  Then there is a natural long exact sequence
of abelian groups 
$$
\cdots \rightarrow \bigoplus_{m \ge 1} \bigoplus_{j \in e' I_p} W_{s- \nu} 
\Omega_{(B[t], (t))}^{i - 2m}
\xrightarrow{e' V^\nu} \bigoplus_{m \ge 1} \bigoplus_{j \in I_p} W_{s} \Omega_{(B[t], (t))}^{i - 2m}$$
$$\xrightarrow{} \RNK_{i-1}(B[x]/(x^e) ) \xrightarrow{\partial} 
\bigoplus_{m \ge 1} \bigoplus_{j \in e' I_p} W_{s- \nu} \Omega_{(B[t], (t))}^{i - 1 - 2m}\xrightarrow{}  \cdots
$$
in which $s   = s(m,j)$ is the unique integer such that $p^{s-1} j \le me < p^s j$ and the group
$W_d \Omega_{(B[t], (t))}^a$ of relative de Rham-Witt differentials is defined to be $\{0\}$ if $d < 0$ or $a < 0$.  

Assume first that $i\geq 3$, so that $\lfloor i/2 \rfloor$ and $\lfloor (i-1)/2 \rfloor$ are both $\geq 1$. 

We find that the only $m$ that can contribute to the terms immediately to the left of the group $\RNK_{i-1}(B[x]/(x^e) )$ in the above sequence are those in the range
$1 \le m \le \lfloor i/2 \rfloor$.  In order for $s = s(m,j)$ to be non-negative for some $j \in I_p$ we must have $p^{-1} j \le p^{s-1} j \le me   < p^s j$.    Now $p^{s-1} j \le me$ implies
 $s \le  \mathrm{log}_p(me/j) + 1 \le \mathrm{log}_p(me)+1 $.  We now observe that since $B$ is an $\mathbb{F}_p$-algebra, the group $W_s \Omega_{(B[t], (t))}^{i - 2m}$ is a module for the Witt vectors $W_s(\mathbb{F}_p)$.  Hence $p^s$ annihilates $W_s \Omega_{(B[t], (t))}^{i - 2m}$. We conclude from this that the highest power of $p$ less than or equal to 
 $$
 p^{\mathrm{log}_p(me)+ 1} = pme
 $$
 is an annihilator of the term $W_s \Omega_{(B[t], (t))}^{i - 2m}$ immediately to the left of the group  $\RNK_{i-1}(B[x]/(x^e))$ in the above 
 long exact sequence.  Since $m \le \lfloor i/2 \rfloor$, the exponent of this power is bounded by
 $\lfloor \mathrm{log}_p(e\lfloor i/2\rfloor ) \rfloor +  1$.

 We now apply a similar analysis to the terms to the right of the group $\RNK_{i-1}(B[x]/(x^e) )$.  Only the terms with $1 \le m \le \lfloor (i-1)/2 \rfloor$ can contribute. Since $j \in e'I_p$ in this case we have $j \ge e'$.  So
 $$
 s - \nu \le  \mathrm{log}_p(me/j) + 1 - \nu \le \mathrm{log}_p(me/e')+1 - \nu =  
 \mathrm{log}_p(m)  + 1.
 $$  
Hence the  highest power of $p$ less than or equal to $pm$ is an annihilator for these terms.  The exponent of this power is bounded by $\lfloor \mathrm{log}_p(\lfloor (i-1)/2\rfloor) \rfloor + 1$.

 For $i=1$, $\RNK_{0}(B[x]/(x^e) )=(0)$. For $i=2$, the terms on the right of $\RNK_{1}(B[x]/(x^e) )$ are trivial
 while the terms immediately to the left can be analyzed exactly as above.
 
Putting these results together we find that the middle  $\RNK_{i-1}(B[x]/(x^e) )$ in the above 
 long exact sequence is annihilated by $p^z$ when $z$ is as in (\ref{eq:zdef}), for all $i\geq 1$.
 \end{proof}

 \subsection{Exponents for $\RNK_{n+1}(M_G\otimes_\BZ R/p)$ and $\RNK_n(R[G]/p)$.}
 
 We now use Proposition \ref{lem:dumbbound}  to obtain information about these groups.  

Set
\[
b(p, n)=\lfloor \log^*_p(\lfloor (n+1)/2\rfloor ) \rfloor  + \lfloor \log^*_p(\lfloor n/2\rfloor) \rfloor.
\]
We have $b(p,0)=-2$, $b(p,1)=-1$, and $b(p, n)=0$ if $2\leq n< 2p-1$. In particular,
 $b(p, n)=0$ if $n\geq 2$ and $p> 1/2+n/2$.

\begin{corollary}\label{cor:truncated}
Suppose that $G$ is cyclic of order $p^m $, $m\geq 1$, and $R/p$ is regular and Noetherian. If $n\geq 0$, the groups
$\RNK_{n+1}(M_G\otimes_\BZ R /p )$  and $\RNK_n(R[G]/p)$ are annihilated by 
$p^{m+2+b(p, n+1) }$
and $p^{m+2+b(p, n) }$, respectively.  
\end{corollary}

\begin{proof}  As noted above,  $M_G\otimes_\BZ R /p$ is a product of finitely many rings each isomorphic to $R/p[x]/(x^e)$ with $e = 1$ or $e = p^\nu \cdot (p-1)$ for some $0 \le \nu \le m-1$. Also  $R[G]/p$ is isomorphic to $R/p[x]/(x^{p^m})$. Applying  Proposition 
\ref{lem:dumbbound}, and using the obvious inequality $\lfloor \log_p(p-1)+\alpha\rfloor \leq 1+\lfloor \alpha\rfloor$, for $\alpha\geq 0$, now gives the result.
\end{proof}

\begin{Remark}\label{rem:smallerexpon}
{\rm  For $n=1$, Proposition 
\ref{lem:dumbbound} actually gives that  the group $\RNK_{2}(M_G\otimes_\BZ R /p)$ is annihilated by the smaller power $p^{m+1 }$.
 For $n=0$, $\RNK_{1}(M_G\otimes_\BZ R /p)$ is annihilated by  $p^{m }$ and $\RNK_{0}( R[G] /p)=(0)$.}
\end{Remark}

From Corollary \ref{cor:truncated}, \eqref{pinF/p} and the above remark, we obtain

\begin{corollary}\label{cor:truncated2}
The group $\pi_n(\RNF(/p))$ is annihilated by
$
 p^{2m+4+b(p, n)+b(p, n+1) }  
$
if $n\geq 2$, and by the smaller power $p^{2m+2}$, if $n=1$, resp. $p^m$, if $n=0$.
If $p\geq 3+n/2$, the group $\pi_n(\RNF(/p))$ is annihilated by $p^{2m+4}$.\qed
\end{corollary}

\section{Cyclic homology and the $p$-adic Beilinson isomorphism}\label{s:Beilinson}

Here we realize (A) of \S \ref{ss:strategy}. Our main tool is the ``$p$-adic Beilinson isomorphism" of \cite{NikolausBFSq} between $p$-adic relative $\RK$-theory and $p$-adic relative cyclic homology.

\subsection{The $p$-adic Beilinson isomorphism}\label{ss:cont}

Assume $S$ is a commutative ring which is $p$-adically complete and $p$-torsion free.
Recall the definition
\[
\RK^c(S, (p))=\varprojlim_a \RK(S/p^a, (p)).
\]

The following is a corollary of the work in  \cite{henselian}  and \cite{NikolausBFSq}. 
 
\begin{theorem}\label{masterprop} 
Let $S$ be a ring as above. There is a finite zig-zag of quasi-isogenies of spectra, in the sense of \cite[\S 2.3]{NikolausBFSq},
between $\RK^c(S, (p))$ and $\Sigma\HC(S, (p);\BZ_p)$, which are natural in $S$. Each quasi-isogeny induces an equivalence of $(2p-5)$-truncations
and hence the zig-zag induces an equivalence
\[
\tau_{\leq 2p-5}\,\RK^c(S, (p))\xrightarrow{\sim} \tau_{\leq 2p-5}\,\Sigma\HC(S, (p);\BZ_p)
\] 
after $(2p-5)$-truncation. As a result, when $n\leq 2p-5$, there is an isomorphism
\[
 \RK^c_n(S, (p))\xrightarrow{ \sim\ }{\rm HC}_{n-1}(S, (p);\BZ_p)
\]
 which is natural in $S$ as above.
\end{theorem}

\begin{proof}
This follows from    \cite[Remark 2.17]{NikolausBFSq} and \cite[Theorem 2.20]{NikolausBFSq} and its proof, see also \cite[Corollary B]{NikolausBFSq}. (As explained in \emph{loc. cit.}, this also uses the cyclotomic trace theorems of Dundas--Goodwillie--McCarthy and results of Clausen--Mathew--Morrow \cite{henselian} on continuity of topological cyclic homology and $\RK$-theory.)  
\end{proof}

\begin{Remark}
{\rm The isomorphism is an integral refinement of a similar $\Q_p$-rational isomorphism in \cite{Be14}.}
\end{Remark}

\begin{Remark}\label{rem:smallprimes}
{\rm The effect of this finite zig-zag of quasi-isogenies on homotopy groups will be described in Proposition \ref{prop:smallprimes}. The behavior of the quasi-isogenies at homotopy groups of degree $n>2p-5$  is more complicated and to study them one has to examine carefully the arguments of
\cite{NikolausBFSq}. The main point, as stated in Remark 2.25 in \emph{loc. cit.}, is that the ``denominators" of the quasi-isogenies in degree $\leq n$ are bounded by a number which is independent of $S$ and only depends on $n$ and $p$, see also below. We thank A. Mathew for some very helpful email correspondence clarifying this point.  }
\end{Remark}

\subsubsection{A relative version}
We now apply this to $S=\wh \A$ and $S=\wh \M $ and use the naturality of the quasi-isogenies for the inclusion $f: \wh \A\to \wh \M $. In fact, it is more convenient to apply Theorem \ref{masterprop}  for this relative situation as follows.

Let $\RK^c(f, (p))=\RK^c(\wh \A, \wh \M, (p))$ be  the mapping fiber of the map of relative continuous $\RK$-theory spectra $\RK^c(\wh \A, (p))\to \RK^c(\wh \M, (p))$ induced by $f$, and similarly for 
 $ \HC(f, (p);\BZ_p)$. The fiber $\HC(f, (p);\BZ_p)$ will be described more explicitly in the next paragraph. The natural zig-zag of quasi-isogenies obtained by Theorem \ref{masterprop} also induces a  zig-zag of quasi-isogenies 
between the fibers $\RK^c(f, (p))$ and $\Sigma\HC(f, (p);\BZ_p)$. These are now 
equivalences after $(2p-6)$-truncation. 
This zig-zag gives   isomorphisms
\begin{equation}\label{relativeBeil}
 \RK^c_n(f, (p))\xrightarrow{ \sim\ }{\rm HC}_{n-1}(f, (p);\BZ_p)
\end{equation}
for all $n\leq 2p-6$. 
Hence, we have isomorphisms
\begin{equation}\label{isopinHC}
  \pi_n(\RF(-, (p)))\xrightarrow{ \sim\ } {\rm HC}_{n-1}(f, (p);\BZ_p)
\end{equation}
for all $n\leq 2p-6$, i.e. for $p\geq 3+n/2$.

To explain the situation at all primes, including small primes, we fix $p$ and $n$ and consider $\RK^c_n(f, (p))$ and  $\HC_{n-1}(f , (p);\BZ_p)$ as functors from 
$p$-adically complete torsion free $\BZ_p$-algebras $\wh R$ to abelian groups. Indeed, this makes sense since $f: \wh \A\to \wh \M$ is the base change of $\BZ[G]\to M_G$
to $\wh R\lan t\ran$.

Using Remark \ref{rem:smallprimes} and the above we see that, in general, the following holds, as a corollary of Theorem \ref{masterprop} and its proof.

\begin{proposition}\label{prop:smallprimes} There is an integer $N$, and a collection of 
functors $\calF_i$, for $i=0,\ldots, N$, and $\calG_i$, for $i=0,\ldots, N-1$, from 
 $p$-adically complete torsion free $\BZ_p$-algebras $\wh R$ to abelian groups, together  with natural transformations 
 \[
 \calF_i\leftarrow \calG_i\to \calF_{i+1},
 \]
  for $i=0,\ldots , N-1$, such that 
 \begin{itemize}
     \item $\calF_0(\wh R )=\RK^c_n(f, (p))$ and $\calF_{N}(\wh R )=\HC_{n-1}(f , (p);\BZ_p)$, as functors,
     \item there are integers $M^-_i(p, n)$, resp. $M^+_i(p, n)$, which are independent of $G$, such that for all $i$, and all $\wh R$, both the cokernel and kernel of
     \[ 
      \calF_i(\wh R)\leftarrow \calG_i(\wh R), \ \ \hbox{\rm resp. }\ \ \calG_i(\wh R)\to \calF_{i+1}(\wh R),
     \] 
     are annihilated by $p^{M^-_i(p, n)}$, resp. $p^{M^+_i(p, n)}$.
 \end{itemize}
 If $n\leq 2p-6$, then we can take $N=1$, $\calG_0=\calF_0$, and $M^-_0(p, n)=M^+_0(p, n)=0$. \qed
 \end{proposition}

 \begin{Remark}
{\rm  Precise values for $M^\pm_i(p, n)$ seem difficult to pin down. The $M^\pm_i(p, n)$ are related, among other things, to the $p$-power torsion in the stable homotopy groups of spheres. On the other hand, one can expect that an upper bound for $M^\pm_i(p, n)$ can be obtained from an effective comparison of syntomic cohomology with derived de Rham cohomology at the integral level. Such a comparison result has recently been announced by J. Lurie and we thank A. Mathew for bringing this to our attention.}
\end{Remark}

\subsection{$p$-adic relative   cyclic homology}

 We first review some facts about ($p$-adic, relative) cyclic homology and discuss complexes which are used for its calculation. We then provide an annihilator for the relative group
$\HC_{n-1}(f, (p);\Z_p)$.

\subsubsection{Cyclic complexes} 

Suppose $S$ is a (commutative) $\BZ_p$-algebra which is $p$-adically complete and $p$-torsion free, i.e. flat over $\BZ_p$.
 (For our application, we only need to consider $S=\wh \A$ or $S=\wh \M $.)

Note that, for a cyclic object $Z$ in a general abelian category, one can define the cyclic homology chain complex $C_*(Z)$, see \cite[\S I]{GoodTop}. The homology of $C_*(Z)$ is 
the cyclic homology of $Z$. A basic example of a cyclic object in the category of $\BZ$-modules is given by 
\[
Z_n(S)=S^{\otimes (n+1)}=S\otimes_{\BZ} S\otimes_{\BZ}\cdots \otimes_{\BZ}S,
\]
and the standard face, degeneracy and cyclic action maps (\cite{LQ}, \cite{GoodTop}).
Then $\HC(S)=C_*(Z(S))$ is the  complex calculating the 
cyclic homology of $S$. (So, ${\rm HC}(S)$ is also the total complex ${\rm Tot}{\mathscr C}(S)$ associated to the double complex ${\mathscr C}(S)$   in \cite{LQ}; here, we mostly follow the notation of \cite{GoodTop}.) 

Next, we consider the derived $p$-adic completions 
$\wh Z(S)$ of $Z(S)$ and $\wh{\rm HC}(S)$ of ${\rm HC}(S)$:
\[
\wh Z(S):={\rm homlim}_m (Z(S)\otimes^{\BL}\Z/p^m), \quad \wh{\rm HC}(S):={\rm homlim}_m ({\rm HC}(S)\otimes^{\BL}\Z/p^m).
\]
 Since $S$ is $\BZ$-flat,  $Z(S)\otimes^{\BL}\Z/p^m$ has terms
 \[
\otimes^{n+1}_{\BZ/p^m}S/p^m= (S/p^m\otimes_{\BZ/p^m}   \cdots \otimes_{\BZ/p^m}S/p^m)=Z_n(S)\otimes_{\Z_p}\Z/p^m.
 \]
 Since the connecting maps $\otimes^{n+1}_{\BZ/p^{m+1}}S/p^{m+1}\to \otimes^{n+1}_{\BZ/p^{m}}S/p^{m}$ are surjective, 
 the derived limit $\wh Z(S)$ of $Z(S)\otimes^{\BL}\Z/p^m$ coincides with the naive limit which has terms
 \[
\wh Z_n(S)=\varprojlim_m(Z_n(S)\otimes_{\Z_p}\Z/p^m)=\varprojlim_m(S/p^m \otimes_{\BZ/p^m}\cdots \otimes_{\BZ/p^m}S/p^m)=
\]
\[
=S\hat\otimes_{\BZ_p}\cdots \hat\otimes_{\BZ_p}S=S^{\wh\otimes (n+1)}.
\]
(Here $\wh\otimes$ denotes the $p$-adically completed tensor product. Recall  that $S$ is   $p$-adically complete.) The $p$-adic completion $\wh Z(S)$ is a cyclic object in the category of $\BZ_p$-modules. The corresponding cyclic homology chain complex $C_*(\wh Z(S))$ is   also the $p$-adic completion $\wh{{\rm {HC}}}(S)$ of $\HC(S)$ and represents $\HC(S; \BZ_p)$. In particular, we have
\[
\HC_*(S; \BZ_p):={\rm H}_*(\wh{{\rm {HC}}}(S)).
\] 

 To give a complex representing the relative cyclic homology $\HC(S, (p);\BZ_p)$ we consider the DGA
 $\calS:=S_0\oplus S_1$
given by $S_0=S$, $S_1=S\cdot x$, with multiplication satisfying $x^2=0$ and differential $\partial: S_1\to S_0$ given by $\partial(x)=p$.  This DGA is a derived replacement for $S/p$.  Then we first have the cyclic object
 $\wh Z(S/p)$ in the category of complexes of $\BZ_p$-modules with terms the complexes
 \[
 \wh Z_n(S/p)=\wh Z_n(\calS, \partial)=(S_0\xleftarrow{\partial} S_1)^{\hat\otimes (n+1)}.
 \] 
These are the $p$-adic completions of
 \[
 Z_n(S/p)= Z_n(\calS, \partial)=(S_0\xleftarrow{\partial} S_1)^{ \otimes (n+1)}.
 \]
 Then, the cyclic object $\wh Z( S, (p))$ is the mapping fiber of the natural map
\[
\wh Z(S)\to \wh Z(S/p)=\wh Z(\calS, \partial),
\]
 and following \cite{GoodTop}, we can consider the complex $C_*(\wh Z(\calS, \partial))$ in the category of complexes of $\BZ_p$-modules.
 The relative cyclic homology  $\HC(S, (p);\BZ_p)$ is now 
represented by the total complex ${\rm Tot}\, C_*(\wh Z( S, (p)))$.

\subsubsection{Relative cyclic homology} 

We now continue with our usual assumptions and notations.
In particular, $G$ is a cyclic group of order $p^m$, $R$ is as in Theorem \ref{thm:MainIntro}, and $f$ is the inclusion $\wh \A\to \wh \M$. 
We will apply  the above to $S=\wh \A$ and $S=\wh\M$ and show the following.

 \begin{proposition}\label{lemma:HCconductor}
With the above notations, for all $n\geq 0$, we have
\[
p^{(n+1)m}\cdot \HC_{n-1}(f, (p);\Z_p)=(0).
\]
\end{proposition}

\begin{proof} We use the constructions of the previous subsection.
The inclusion $f: \wh \A\to \wh \M$
 induces a map of cyclic objects in the category of complexes of $\BZ_p$-modules
 \[
 \wh Z(\wh \A, (p))\to  \wh Z(\wh \M, (p))
 \]
 (i.e. respecting all the structures) and we let 
 \[
\wh Z(f, (p)):={\rm Fib}(\wh Z(\wh \A, (p))\to  \wh Z(\wh \M, (p)))
\]
be the mapping fiber of this map.

 Note that the terms of the complexes $\wh Z(\wh \A, (p))$ and $\wh Z(\wh \M, (p))$ are abstractly isomorphic, as $\BZ_p$-modules, to direct sums of copies of the   completed tensor products $(\wh \A)^{\hat\otimes k}$ and $(\wh \M)^{\hat\otimes k}$, respectively.  Under these decompositions, the induced map
 $\wh Z(\wh \A, (p))\to  \wh Z(\wh \M, (p))$ is a direct sum of the maps $f^{\hat \otimes k}$.
 Hence, Lemma \ref{lemma:incl} below implies that 
the map $\wh Z(\wh \A, (p))\to  \wh Z(\wh \M, (p))$  is injective on all terms of the complexes. Assuming Lemma \ref{lemma:incl} for the moment, we now complete the proof of the proposition.
The injectivity implies that
\[
{\rm Fib}(\wh Z(\wh \A, (p))\to  \wh Z(\wh \M, (p)))= \left(\wh Z(\wh \M, (p))/f(\wh Z(\wh \A, (p)))\right)[1].
\]
Hence, all the terms of  $\wh Z(f, (p)) $ are direct sums of $\BZ_p$-modules of the form 
$ (\wh \M)^{\hat\otimes k}/f^{\hat\otimes k}( (\wh \A)^{\hat\otimes k})$. Since $p^mM_G\subset \BZ_p[G]$ by Lemma \ref{lemma:conductor}, we have
$p^m \wh \M\subset \wh \A$ and so
\begin{equation}\label{pmkills}
p^{km}\cdot (\wh \M)^{\hat\otimes k}/f^{\hat\otimes k}( (\wh \A)^{\hat\otimes k})=(0).
\end{equation}
The mapping fiber $ \wh Z(f, (p))$  is again a cyclic object in the category of complexes of $\BZ_p$-modules and we have an exact sequence
\[
0\to {\rm Tot}\, C_*(\wh Z(\wh \A, (p)))\to  {\rm Tot}\, C_*(\wh Z(\wh \M, (p)))\to {\rm Tot}\, C_*(\wh Z(f, (p)))[-1]\to 0
\]
 (\cite[Prop. 2.7]{GoodTop}). Hence, the relative cyclic homology $\HC(f, (p);\BZ_p)$ is represented by ${\rm Tot}\, C_*(\wh Z(f, (p)))$. By the above,  
$\HC_{n-1}(f, (p);\BZ_p)$ is a subquotient of a direct sum of copies of $(\wh \M)^{\hat\otimes k}/f^{\hat\otimes k}( (\wh \A)^{\hat\otimes k})$,
for $1\leq k\leq n+1$. But by \eqref{pmkills}, this is annihilated by $p^{(n+1)m}$.
\end{proof}
 
  \begin{lemma}\label{lemma:incl}
 For each $k\geq 1$, the homomorphism
 \[
 f^{\hat \otimes k}: (\wh \A)^{\hat\otimes k}\to (\wh \M)^{\hat\otimes k}
 \]
 induced by $f$ is injective.
 \end{lemma}
 
 \begin{proof} Observe that, since $\BZ[G]$ and $M_G$ are free finitely generated abelian groups, we have
 \[
  (\wh \A)^{\hat\otimes k}=(\BZ[G])^{\otimes k}\otimes_{\BZ} \wh R\lan t\ran^{\hat\otimes k}, \quad   (\wh \M)^{\hat\otimes k}=(M_G)^{\otimes k}\otimes_{\BZ} \wh R\lan t\ran^{\hat\otimes k},
 \]
where $\wh R\lan t\ran$ is the $p$-adic completion of $R[t]$. If $\iota: \BZ[G]\to M_G$ denotes the inclusion, then
\[
\iota^{\otimes k}: (\BZ[G])^{\otimes k}\xrightarrow{\ } (M_G)^{\otimes k}
\]
is injective and 
$ f^{\hat \otimes k}=\iota^{\otimes k}\otimes_{\BZ}\wh R\lan t\ran^{\hat\otimes k}$.
Hence, it is enough to show that $\wh R\lan t\ran^{\hat\otimes k}=\wh R\lan t\ran\hat\otimes_{\BZ_p} \cdots \hat\otimes_{\BZ_p}\wh R\lan t\ran$ is  flat over $\BZ_p$ and hence over $\BZ$. But
\[
\wh R\lan t\ran^{\hat\otimes k}=\varprojlim_a  R[ t]^{\otimes k}/p^a=\varprojlim_a \left((R[t]/p^a)\otimes_{\BZ/p^a}\cdots\otimes_{\BZ/p^a} (R[t]/p^a)\right),
\]
 and $(R[t]/p^a)^{\otimes k}$ is flat over $\BZ/p^a$, since $R[t]/p^a$ is $\BZ/p^a$-flat. The flatness of $\wh R\lan t\ran^{\hat\otimes k}$ over $\BZ_p$ now follows by
\cite[Lemma 15.27.4]{StacksProj}.
 \end{proof}

 \begin{corollary}\label{cd2} Assume $G$ is a cyclic group of order $p^m$ and $R$ satisfies the assumptions of Theorem \ref{thm:MainIntro}. 
 There exists an integer $c(p, n)$, which is independent of $R$ and $G$, such that 
\[
p^{(n+1)m+c(p, n)}\cdot \pi_n(\RF(-, (p))) =(0).
\]
If $p\geq 3+n/2$,  then $c(p, n)=0$. \end{corollary}
\begin{proof}
Set $c(p,n)=\sum_{i=0}^{N-1}(M^+_i(p, n)+M^-_i(p,n))$. Then this follows from Proposition \ref{prop:smallprimes} and Proposition \ref{lemma:HCconductor}.
\end{proof}

\subsection{Proof of Theorem \ref{thmexpo}}\label{ss:Finale} We can finally assemble all the ingredients following the outline in \S \ref{ss:strategy}. By combining Corollary \ref{cd2}, Corollary 
\ref{cor:truncated2} and the exact sequence \eqref{pinFexact}, we see that $\pi_n(\RNF)\simeq \RNK_n(R[G])$ is annihilated by
\[
p^{(n+3)m+4+a(p,n)},
\]
where we set
\[
a(p, n):=b(p, n+1)+b(p,n)+c(p,n).
\]
 Theorem \ref{thmexpo} now follows. \qed

\subsection{The case $n=1$}\label{ss:n1case} Under some additional conditions on $R$, we can give somewhat better bounds for the exponent of $\RNK_1(R[G])$.
To keep the presentation simple, we just outline the argument when $R=\BZ$   and omit the details. Assume $G=C_{p^m}$ is cyclic of order $p^m$. 
Then, using \cite[Theorem 1.2, Theorem 1.3]{CPTII}, we can obtain 
\[
{\bf R}^1\prolim_a \RK_2(\A/p^a)=(0),\quad  \prolim_a \RK_1(\A/p^a)\simeq \RK_1(\wh \A).
\]
This, combined with the Milnor exact sequence, gives $\RK^c_1(\wh \A)\simeq \RK_1(\wh \A)$. Similarly we obtain $\RK^c_1(\wh \M)\simeq \RK_1(\wh \M)$. By 
\cite[Lemma 3.8]{CPTII} we have $\rSK_1(\wh \A)=(0)$ and now this quickly gives $\ker(\RK^c_1(\wh \A)\to \RK^c_1(\wh \M))=(0)$.
 We can now use \eqref{MV2} and the discussion just below it, to see that it is enough to find an exponent for the cokernel of $\RK^c_2(\wh \A)\to \RK^c_2(\wh \M)$. We then follow the same strategy, using localization at $(p)$, as in \S \ref{ss:strategy}. In this case, because of the above,   there is no need to bound $\RK_1(\A/p)$ and we can use Remark \ref{rem:smallerexpon} to bound $\RK_2(\M/p)$. This way we can ``shave off" $m+3$ from the exponent. Using the same bound as before for the relative to $(p)$ part, we can now conclude that
 \[
 p^{3m+1+a(p,1)}\cdot \RNK_1(\BZ[C_{p^m}])=(0).
 \]
This leads to the result that
\[
e(G)^3|G|\prod_{p||G|}  p^{a(p,1)} \cdot \RNK_1(\BZ[G])=(0),
\]
 for an arbitrary finite group $G$ which improves on Theorem \ref{thm:general} in this case.

 \section{Boundedness of Farrell Nil groups}\label{s:AppendixB}

 \subsection{Farrell Nil groups} For a unital associative ring $A$ with an automorphism $\alpha: A\to A$, we denote by $A_\alpha[t]$ 
  the skew polynomial ring with coefficients in $A$, variable $t$, and multiplication satisfying $t\cdot a=\alpha(a)t$, for all $a\in A$. 
  By definition, the Farrell Nil groups of $(A,\alpha)$  are
  \[
\RNK_n(A, \alpha):=   \ker(\RK_n(A_\alpha[t])\xrightarrow{\epsilon} \RK_n(A)),
  \]
where $\epsilon$ is obtained by the augmentation $A_\alpha[t]\to A, t\mapsto 0$. We will use:
\begin{theorem}\label{skewhomotopy}
If $A$ is left regular Noetherian then so is $A_\alpha[t]$ and we have $\RNK_n(A, \alpha)=(0)$.
\end{theorem}
\begin{proof} This is standard, see e.g. \cite[Thm. 25]{FaHs}, \cite[Prop. 1.1]{FaJo}, \cite{Grayson}, \cite[Theorem 10.1 and Corollary 10.3]{BaLuck}.\end{proof}

 Let $\alpha: G \to G$ be an automorphism of a finite group $G$. We shall also write $\alpha: R[G] \to R[G]$ for the $R$-linear ring automorphism of $R[G]$ induced by the group automorphism.  
The main result of this section is:

  \begin{theorem}\label{thm1_appB}
Let $G$ be a finite group,  $\alpha: G\to G$ an automorphism, and $R$ a commutative ring which  is regular, Noetherian, flat over $\BZ$, and such that $R/p$ is regular for all primes $p$ that divide $|G|$. Then, for any $n\in\BZ$, the group $\NK_n(R[G],\alpha)$ is annihilated by a power of $|G|$.
\end{theorem}

Note that 
$\NK_n(R[G],{\rm id})=\NK_n(R[G])$ so this result also gives an unspecified finite annihilator for $\NK_n(R[G])$. However, for $\alpha={\rm id}$, Theorem \ref{thm:general} is stronger since it gives an annihilator for $\NK_n(R[G])$ which only depends on $|G|$, $n$, 
and the primes that divide $|G|$ and are non-units in $R$.

\begin{proof}
This follows the main steps in the proof of Theorem \ref{thm1_app}, adjusted to the skew set-up. We thank M.~Morrow for suggesting a simplification to our original argument. 

For simplicity, we set $q=|G|$. By \cite[Theorem 1.2]{FaJo}  there exists a $\BZ$-order $\M$ in $\Q[G]$ which contains $\BZ[G]$, is $\alpha$-invariant, and is a hereditary  ring, hence left/right regular.
In addition, we have $q\M\subset \BZ[G]$.

\begin{proposition}\label{propB}
Under the above assumptions on $R$,  $\M\otimes_\BZ R$ is left regular Noetherian.
\end{proposition}

\begin{proof}
The ring $\M\otimes_\BZ R$ is left (and right) Noetherian. Since $\M[1/q]=\BZ[G][1/q]$ we can easily see that 
$(\M\otimes_\BZ R)[1/q]=R[G][1/q]$ is left regular: Indeed, suppose $M$ is a finitely generated $R[G][1/q]$-module.  Since $q=|G|$, the standard averaging argument shows that $M$ is a direct summand of an induced module $\BZ[G]\otimes_\BZ N$, with $N$ a finitely generated $R[1/q]$-module. But $\BZ[G]\otimes_\BZ N$ has a finite projective resolution obtained by tensoring a finite projective resolution of $N$ which exists since $R[1/q]$ is regular.  It now remains to show that the localization $\M\otimes_\BZ R_{\frak m}$ is left regular, for  every maximal ideal
$\frak m$ of $R$ above a prime $p$ which divides $q=|G|$. By our assumptions on $R$, the prime $p$ is part of a regular system of parameters of $R_{\frak m}$ and 
we can write ${\frak m}=(p,x_1,\ldots, x_d)$, where $d+1=\dim(R_{\frak m})$.  
We  set $S=R_{\frak m}/(x_1,\ldots, x_d)$. This is a regular local ring of dimension $1$,
with uniformizer $p$, i.e. a DVR which is flat and unramified over $\BZ_{(p)}$. 

\begin{lemma}\label{lemmaB1}
 The ring $\M\otimes_\BZ S$ is hereditary.
 \end{lemma}
 
 \begin{proof} Let
 $k=S/(p)$ be the residue field of $S$ which is a (separable) extension of the perfect field $\BF_p$. 
 We have $\M\otimes_\BZ S=\M_{(p)}\otimes_{\BZ_{(p)}}S$. Since $\M_{(p)}$ is hereditary, if $S/\BZ_{(p)}$ is in addition finite,
 $\M_{(p)}\otimes_{\BZ_{(p)}}S$ is hereditary by \cite{Jan}. For a general flat and unramified extension $S/\BZ_{(p)}$ of DVRs, 
 a similar proof applies: Let $J:={\rm rad}(\M_{(p)})$ so that $\M_{(p)}/J$ is a finite dimensional semisimple $\BF_p$-algebra. 
  Since  $\M_{(p)}$ is hereditary, $J$ is a projective $\M_{(p)}$-module.
Consider the $p$-adic completion $\hat S$ and set $B:=\M_{(p)}\otimes_{\BZ_{(p)}}\hat S$, $J_B:=J\otimes_{\BZ_{(p)}}\hat S$.
Since $pB\subset J_B$ and $J_B/pB$ is nilpotent, $J_B$ is contained in the Jacobson radical ${\rm rad}(B)$. Now
\[
B/J_B=(\M_{(p)}/J)\otimes_{\BF_p}k,
\]
which is still semisimple. It follows that $J_B={\rm rad}(B)$. Since 
$J$ is a projective $\M_{(p)}$-module, $J_B$ is a projective $B$-module.
Recall that an order over a complete DVR is hereditary if and only if its Jacobson radical is projective, see \cite[Thm. (39.1)]{Reiner}.
Hence, $B=\M_{(p)}\otimes_{\BZ_{(p)}}\hat S$ is hereditary.  We can now conclude by faithfully flat descent along $S\to \hat S$ that $\M_{(p)}\otimes_{\BZ_{(p)}}S=\M\otimes_\BZ S$
is hereditary.
 \end{proof}

We will  need the following
rather standard lemma.
 
 \begin{lemma}\label{lemmaB2}
 Suppose $(A,\frak m)$ is a Noetherian commutative local ring and $x\in \frak m$ a non-zero divisor in $A$. 
 Let $\Lambda$ be a unital associative $A$-algebra which is a finite  free $A$-module. If $\Lambda/x\Lambda$ has finite left global dimension equal to $N$,
 then $\Lambda$ has finite left global dimension equal to $N+1$.
 \end{lemma}
 
 \begin{proof} This follows from \cite[Prop. 5.6]{Ramras} since, under our assumptions, $x$ is also regular in $\Lambda$. 
 \end{proof}

 We can now complete the proof of Proposition \ref{propB}. By Lemma \ref{lemmaB1}, $\M\otimes_\BZ S$  has left global dimension $1$. 
 An inductive argument starting from $S=R_{\frak m}/(x_1,\ldots, x_d)$ and using Lemma \ref{lemmaB2} for the induction step, 
 gives that
 $\M\otimes_{\BZ}R_{\frak m}$ has left global dimension $d+1$ and so is left regular. \end{proof}
 \smallskip

Let us now continue with the proof of Theorem \ref{thm1_appB}. Set $I=q\M\otimes_\BZ R\subset R[G]=\BZ[G]\otimes_\BZ R$. This is a two-sided ideal of $R[G]$ which is mapped isomorphically to the central ideal $J=q(\M\otimes_\BZ R)$ of $\M\otimes_\BZ R$
  under the natural injection $R[G]=\BZ[G]\otimes_\BZ R\hookrightarrow \M\otimes_\BZ R$. By \cite[Theorem 13.3,  Exercise 1ZA(c)]{GoodWar},
 central ideals in Noetherian rings have the Artin-Rees property. Hence, by \cite[Lemma 2.1, Remark 2.2]{Morrowpro2} the pro-Tor groups vanish:
 \[
 \{{\rm Tor}^{\M\otimes_\BZ R}_i((\M\otimes_\BZ R)/J^r, (\M\otimes_\BZ R)/J^r)\}_r=(0),
 \]
for all $i>0$. By \cite[Theorem 0.2]{Morrowpro2}, the ideal $I$ is pro-Tor unital and by work of Geisser-Hesselholt \cite{GHbi}, $\RK$-theory pro-excision holds.
Hence, the bi-relative pro-abelian groups vanish:
\[
\{\RK_n(R[G], \M\otimes_\BZ R, I^r)\}_r=(0),
\]
 for all $n\in \BZ$. A similar argument works for the skew polynomial rings:
We consider $R[G]_\alpha[t]\hookrightarrow (\M\otimes_\BZ R)_\alpha[t]$ and recall that  $I$ is $\alpha$-invariant. This inclusion
maps the two-sided ideal $I_\alpha[t]$ of polynomials with coefficients in $I$ isomorphically to
the ideal $(q)=(q\M\otimes_\BZ R)_\alpha[t]$.  We have $(I_\alpha[t])^r=I^r_\alpha[t]$, for $r\geq 1$, and this maps isomorphically to $(q^r)$ in $(\M\otimes_\BZ R)_\alpha[t]$.
As before, we conclude  that 
\[
\{\RK_n(R[G]_\alpha[t], (\M\otimes_\BZ R)_\alpha[t], I^r_\alpha[t])\}_r=(0),
\]
for all $n\in \BZ$. We now obtain 
 exact sequences of pro-abelian groups
  \begin{multline}
  \cdots\to \RNK_n(R[G],\alpha)\to \RNK_n(\M\otimes_\BZ R,\alpha)\oplus \{\RNK_n(R[G]/I^r, \alpha)\}_r \to\\\to \{\RNK_n((\M\otimes_\BZ R)/q^{r},\alpha)\}_r\to\cdots. 
  \end{multline}
By Proposition \ref{propB} and Theorem \ref{skewhomotopy}, $\RNK_n(\M\otimes_\BZ R,\alpha)=(0)$, so these become exact sequences of pro-abelian groups
\[
  \{\RNK_{n+1}(\M\otimes_\BZ R/q^{r}, \alpha)\}_r   \to \RNK_n(R[G],\alpha)\to  \{\RNK_n(R[G]/I^r, \alpha)\}_r. 
\]
By Lemma \ref{lemma_app}, it is enough to show that, for each $n$ and each $r$, the groups
$\RNK_{n}(R[G]/I^r, \alpha)$ and $\RNK_{n }(\M\otimes_\BZ R/q^{r}, \alpha)$ are annihilated by a power of $q$  (which is allowed to depend on $r$).
Note that $q\BZ[G]\subset q\M$ and so, for each $r$, $R[G]/I^r$, $\M\otimes_\BZ R/q^{r}$, and the corresponding skew polynomial rings are
all annihilated by some power of $q$. Using \cite[Thm. A]{GH2} (Theorem \ref{GH_app}) and applying a devissage similar to the one in the proof of Theorem \ref{thm1_app} we reduce to showing 
that  
\[
\RNK_{n}(R[G]/I', \alpha), \qquad \RNK_{n }(\M\otimes_\BZ R/q', \alpha),
\]
are annihilated
by a power of $q'$, where $q'$ is the product of the distinct prime factors of $q$ and $I'=I+(q')$. We now decompose into a direct sum over these prime factors
and concentrate on one of them, denoted by $p$. Consider the Jacobson radicals $J(\BF_p[G])$ and $J(\M/p)$ 
of $\BF_p[G]$ and $\M/p$, respectively. These are two-sided $\alpha$-invariant nilpotent ideals of these Artinian rings and they generate
corresponding two-sided nilpotent ideals $J_1\subset R[G]/I'$ and $J_2\subset \M\otimes_\BZ R/q'$, respectively.
Apply \cite[Thm. A]{GH2} once again: First to  $\RK$-groups relative to the ideals $J_1$ and $J_2$, and then to $\RK$-groups relative to the ideals generated by $J_1$ and $J_2$ in the 
skew polynomial rings $(R[G]/I')_\alpha[t]$ and $(\M\otimes_\BZ R/q')_\alpha[t]$, respectively. This further reduces us to bounding
\begin{equation}\label{lastNils}
\RNK_{n}(R[G]/(I'+J_1), \alpha), \qquad \RNK_{n }((\M\otimes_\BZ R/q')/J_2, \alpha).
\end{equation}
However, both $R[G]/(I'+J_1)$ and $(\M\otimes_\BZ R/q')/J_2$ are  tensor products of semisimple Artinian rings over $\BF_p$
with $R/p$. Hence, they are both isomorphic to a finite direct product of rings of the form ${\rm Mat}_{m\times m}({\BF}_{p^f})\otimes_{\BF_p}R/p$.
Since $R/p$ is regular, these tensor products are left regular. Hence, the Nil groups \eqref{lastNils} are trivial by Theorem \ref{skewhomotopy}.
\end{proof}
\smallskip

\subsection{$\RK$-theory of virtually cyclic groups}
Recall that a group is called \emph{virtually  cyclic} if it contains a finite index subgroup which is cyclic. Such a group is either finite, or 
\emph{virtually  infinite cyclic}.
This class of groups plays an important role in the $\RK$-theory of group rings. 
Indeed, the Farrell–Jones conjecture predicts, in broad terms, that the algebraic $\RK$-theory of arbitrary group rings is assembled from the contributions of virtually cyclic subgroups.

\begin{theorem}\label{cor2_appB}
For every $n\in \BZ$ and every virtually   cyclic group $\Gamma$, the group $\RK_n(\BZ[\Gamma])$ is isomorphic to a direct sum 
$
\Lambda\oplus (\oplus_\infty T)
$
of a finitely generated abelian group $\Lambda$
with an infinite countable direct sum  of copies of a finite abelian group $T$.
\end{theorem}

\begin{proof} If $\Gamma$ is  finite, $\RK_n(\BZ[\Gamma])$ is finitely generated   by results of Kuku \cite{Kuku2} and Carter \cite{Carter}. 

Suppose now that $\Gamma$ is infinite, so it
is virtually infinite cyclic. There are two types of virtually infinite cyclic groups $\Gamma$, see \cite[Lemma 3.31]{DKR}:
\begin{itemize}
\item[i)]   $\Gamma\simeq G\rtimes_\alpha \BZ$, where $G$ is a finite group and $\BZ$ acts on $G$ by an automorphism $\alpha: G\to G$. 

\item[ii)] $\Gamma\simeq G_0\star_H G_1$, an amalgamated free product, where $G_0$, $G_1$ are finite groups and $H$ is a subgroup of $G_0$ and of $G_1$ with $[G_0:H]=[G_1:H]=2$.
In this case, $\Gamma$ surjects  to the infinite dihedral group $D_\infty=\BZ\rtimes\BZ_2 $ with finite kernel. This gives a canonical index $2$ subgroup $\ov\Gamma\subset \Gamma$
which is of type (i), in fact, $\ov\Gamma\simeq H\rtimes_\alpha \BZ$, with $\alpha: H\to H$ an automorphism, see \cite{DKR}.
\end{itemize}

The algebraic $\RK$-theory decomposition theorems of Waldhausen (see for example \cite[(1), (2), p. 2392]{DKR}), combined with \cite[Thm 0.1]{DKR},
express $\RK_n(\BZ[\Gamma])$ as a direct sum of two abelian groups: The first  group involves $\RK$-groups of integral group rings of \emph{finite} groups ($\BZ[G]$ in type (i), or $\BZ[G_0]$, $\BZ[G_1]$, $\BZ[H]$, in type (ii), with   notations as above). The second group is either isomorphic to a Farrell Nil group (in type (ii)), or to a direct sum of two Farrell Nil groups (in type (i)). By results of Kuku \cite{Kuku2} and Carter \cite{Carter}, the first group is  finitely generated. By Theorem \ref{thm1_appB} for $R=\BZ$ combined with
\cite[Theorem C]{LPW}, the second group is isomorphic  to $\oplus_\infty T$ as above.
\end{proof}

 \subsubsection{The  Whitehead group} The case $n=1$  is of particular interest due to a connection with the Whitehead group. Recall that the Whitehead group ${\rm Wh}(\Gamma)$ of a group $\Gamma$ is the quotient of $\RK_1(\mathbb{Z}[\Gamma])$ by the image of $\{\pm \gamma: \gamma \in \Gamma\} \subset \mathbb{Z}[\Gamma]^*$ (see \cite[2.4.1]{Rosenberg}).  
We obtain:  

\begin{corollary}\label{cor3_appB} Suppose $\Gamma$ is a virtually cyclic group.
The torsion subgroup ${\rm Wh}(\Gamma)_{\rm tor}$ of the Whitehead group ${\rm Wh}(\Gamma)$
has finite   exponent.
\end{corollary}

\begin{proof} This follows quickly from Theorem \ref{cor2_appB} for $n=1$. 
 \end{proof}

\bigskip
 
\appendix
 
 \section{Boundedness of $\NK$-groups via pro-cdh descent}\label{s:Appendix}

 \centerline{\sc by M. Morrow}
 
 \numberwithin{equation}{section}
\newtheorem{theoremA}[equation]{Theorem}
\newtheorem{propositionA}[equation]{Proposition}
\newtheorem{lemmaA}[equation]{Lemma}
\newtheorem{conjectureA}[equation]{Conjecture}
\newtheorem{remark}[equation]{Remark}

\bigskip

I am grateful to Ted Chinburg, George Pappas, and Martin Taylor for offering me this opportunity to record what I know about the following conjecture, in which we fix a prime number $p$:

\begin{conjectureA}\label{conjecture_app}
Let $X$ be a Noetherian scheme such that $X\otimes_{\bbZ}{\bbZ}[1/p]$ is regular. Then, for any $n\in\bbZ$, the group \[\NK_n(X):=\op{Ker}(\RK_n(\bbA_X^1)\to \RK_n(X))\] is annihilated by a power of $p$.
\end{conjectureA}

Recall that according to a classical result of Weibel \cite[Corol.~6.4]{WMV} one has \[\NK_n(X)[1/p]=\NK_n(X\otimes_\BZ\BZ[1/p]),\] which vanishes since $X\otimes_\BZ\BZ[1/p]$ is assumed to be regular. The conjecture is predicting that more is true, namely that the torsion group $\NK_n(X)$ has finite $p$-power exponent (possibly depending on $n$).

If memory serves, this conjecture dates back to conversations with Moritz Kerz and Georg Tamme during a visit to the University of Regensburg in 2015. I have also benefited from discussions about it at least with Ben Antieau, Elden Elmanto, and Akhil Mathew. In any case, although some of the results in this appendix are not explicitly recorded in the literature, a number of the arguments are likely ``well-known to experts.''

The statements and arguments below work for \[\mathrm{N}^m\RK_n(X):=\op{Ker}(\RK_n(\bbA_X^m)\to \RK_n(X))\] for any $m\ge 1$. We restrict to $m=1$ for simplicity of notation.

We begin with the case of interest for the main article, explaining how pro-excision and a theorem of Geisser--Hesselholt can be used to give a short proof of the following result; it shows that the NK-groups of Theorem \ref{thmexpo} are of finite exponent but a priori gives no indication of an explicit bound:

\begin{theorem}\label{thm1_app}
Let $G$ be a cyclic $p$-group, and $R$ a commutative ring satisfying the assumptions of Theorem \ref{thm:MainIntro}, i.e., $R$ is regular, Noetherian, flat over $\BZ$, and $R/p$ is regular. Then, for any $n\in\BZ$, the group $\NK_n(R[G])$ is annihilated by a power of $p$.
\end{theorem}

We will systematically use pro-abelian groups and so offer a short reminder on them; the reader can consult \cite{ArtinMazur1969}, \cite{Isaksen2002}, or \cite{Morrowpro2} for further details.

Given a category $\cal C$ (for us it will always be an abelian category), the associated category $\op{Pro}\cal C$ of countable pro-objects is defined as follows: its objects consist of inverse systems $\cdots\to A_2\to A_1$ in $\cal C$, denoted by $\{A_r\}_r$, and morphisms in $\op{Pro}\cal C$ are given by \[\Hom_{\op{Pro}\cal C}(\{A_r\}_r,\{B_s\}_s):=\varprojlim_s\indlim_r\Hom_{\cal C}(A_r,B_s).\] We refer to either of the first two references above for the unsurprising definition of composition. There is a fully faithful embedding $\cal C\to\op{Pro}\cal C$ by sending an object of $\cal C$ to the associated constant inverse system.

Now suppose that $\cal C$ is an abelian category. Then so is $\op{Pro}\cal C$, and kernels and cokernels are computed term-wise when it makes sense. An object $\{A_r\}_r\in\op{Pro}\cal C$ is isomorphic to zero if and only if for each $r\ge 1$ there exists $s\ge r$ such that the transition map $A_s\to A_r$ is zero. Such inverse systems are often called Mittag--Leffler zero, and in this case $\op{Pro}\cal C$ can be alternatively defined as the abelian category of countable inverse systems modulo the Serre subcategory of Mittag--Leffler zero systems.

Taking $\cal C=\textrm{Ab}$ to be abelian groups, we obtain the category of {\em pro-abelian groups} $\op{Pro}\textrm{Ab}$. The following simple lemma (which in more conceptual terms concerns a compatibility between Pro and Serre subcategories of abelian categories) will be crucial:

\begin{lemma}\label{lemma_app}
Let $B$ be an abelian group, $\{A_r\}_r$ and $\{C_r\}_r$ pro-abelian groups, and \[\{A_r\}_r\stackrel f\To B\stackrel g\To \{C_r\}_r\] an exact sequence in $\op{Pro}\mathrm{Ab}$. Assume, for each $r\ge1$, that $A_r$ and $C_r$ are annihilated by some power of $p$ (possibly depending on $r$). Then $B$ is annihilated by some power of $p$.
\end{lemma}
\begin{proof}
The morphism $f$ is represented by a map of abelian groups $f_t:A_t\to B$ for some $t\ge 1$ which we fix, while the morphism $g$ is represented by compatible maps of abelian groups $g_r:B\to C_r$ defined for all $r\ge1$. Exactness at the middle of the complex means that the decreasing towers of subgroups of $B$
\[\op{Ker}(g_1)\supseteq \op{Ker}(g_2)\supseteq\op{Ker}(g_3)\supseteq\cdots\]
\[\op{Im}(f_t)\supseteq \op{Im}(A_{t+1}\to A_t\xto{f_t} B)\supseteq \op{Im}(A_{t+2}\to A_t\xto{f_t} B)\supseteq\cdots\]
are intertwined. In particular, there exists $r\ge t$ such that $\op{Im}(f_t) \supseteq \op{Ker}(g_r)$. But our hypotheses imply that $\op{Im}(f_t)$ and $B/\op{Ker}(g_r)$ are annihilated by a power of $p$, whence the same is true of $B$.
\end{proof}

Next we recall:

\begin{theorem}[Geisser--Hesselholt]\label{GH_app}
Let $A$ be a ring in which $p$ is nilpotent, and let $I\subseteq A$ be a nilpotent ideal. Then, for any $n\in\bb Z$, the relative $\RK$-group $\RK_n(A,I)$ is annihilated by a power of $p$.
\end{theorem}
\begin{proof}
This is \cite[Thm.~A]{GH2}.
\end{proof}

We may now prove Theorem \ref{thm1_app}. In contrast to the main article (see especially the beginning of \S 4.2), we work within the category of pro-abelian groups instead of passing to the limit via the functor $\op{Pro}\mathrm{Ab}\to\mathrm{Ab}$, $\{A_r\}_r\mapsto\varprojlim_r A_r$.

\begin{proof}[Proof of Theorem \ref{thm1_app}]
We adopt the notation of \S4. In particular, $G$ is a cyclic $p$-group of order $p^m$, $R$ a commutative ring satisfying the assumptions of Theorem \ref{thm:MainIntro}, and $M_G$ is the normalisation of $\BZ[G]$ in $\Q[G]$. Recall that
 \[
 p^m R[G]\subseteq p^mM_G\otimes_\bbZ R\subseteq R[G]\subseteq M_G\otimes_\bbZ R
 \]
  and similarly after adding a polynomial variable. Pro-excision for $\RK$-theory \cite[Corol.~0.4]{Morrowpro2} therefore implies that the bi-relative pro-abelian groups $\{\NK_n(R[G],M_G\otimes_\bbZ R, p^{mr}M_G\otimes_\bbZ R)\}_r$ vanish for all $n\in\BZ$, thereby yielding a long exact Mayer--Vietoris sequence of pro-abelian groups 
  \begin{multline}
  \cdots\to \NK_n(R[G])\to \NK_n(M_G\otimes_\bbZ R)\oplus \{\NK_n(R[G]/p^{r})\}_r\to\\ \to \{\NK_n(M_G/p^{r}\otimes_\bbZ R)\}_r\to\cdots. 
  \end{multline}

The regularity of $M_G\otimes_\BZ R$ (Lemma \ref{regular}) implies that $\NK_n(M_G\otimes_\BZ R)=0$, so we obtain exact sequences
 \[
 \{\NK_{n+1}(M_G/p^{r}\otimes_\BZ R)\}_r\To \NK_n(R[G])\To \{\NK_n(R[G]/p^{r})\}_r.
\]
 To appeal to Lemma \ref{lemma_app} to complete the proof, we must check that the NK-groups appearing in the pro-abelian groups on the left and right are annihilated by a power of $p$ (which is allowed to depend on $r$). Applying Theorem \ref{GH_app} to $M_G/p^r\otimes_\bbZ R$, to $R[G]/p^r$, and to their polynomial algebras, the problem then reduces to showing that each $\NK$-group of $(M_G/p^r\otimes_\bbZ R)_{\rm red}$ and $(R[G]/p^r)_{\rm red}$ is annihilated by some power of $p$. But the latter two reduced rings are respectively equal to $(R/p)^{m+1}$ and $R/p$, which are regular, so their NK-groups vanish.
\end{proof}

The previous proof consisted of three steps: (1) use pro-excision to reduce to rings in which $p$ was nilpotent; (2) use Theorem \ref{GH_app} to pass to $\bbF_p$-algebras; (3) get lucky that the $\bbF_p$-algebras $(M_G/p^r\otimes_\bbZ R)_{\rm red}$ and $(R[G]/p^r)_{\rm red}$ were regular and so did not contribute any NK-groups to the problem.

Theorem \ref{thm2_app} below records the strongest general form of Conjecture \ref{conjecture_app} which I know at present how to prove. It is obtained by replacing pro-excision by pro-cdh descent and the good fortune of step (3) by the next theorem.

\begin{theorem}\label{theorem_app_EM}
Let $X$ be a quasi-excellent Noetherian $\bbF_p$-scheme. Then, for any $n\in\bbZ$, the group $\NK_n(X)$ is annihilated by a power of $p$.
\end{theorem}
\begin{proof}
Let $\NK:=\op{fib}(\RK(-)\to\RK(\bbA^1_-))$ as a presheaf of spectra on qcqs $\bb F_p$-schemes. We need the non-trivial fact that $\RK(-;\bbQ_p)$ (i.e., $p$-complete K-theory then invert $p$) is an h-sheaf on qcqs $\bbF_p$-schemes: see the proof of the implication (1)\, $\Rightarrow\, $(2) of 
\cite[Lem.~4.25]{ElmantoMorrow}. Therefore $\RK(\bbA_-^1;\bbQ_p)$ and $\NK(-;\bbQ_p)$ are also $h$-sheaves on qcqs schemes. But $\NK(-;\bbQ_p)$ vanishes on regular Noetherian schemes (as $\NK$ does); so it vanishes on all quasi-excellent Noetherian $\bbF_p$-schemes by the existence of alterations.

For any quasi-excellent Noetherian $\bb F_p$-scheme $X$, we now know that both $\NK(X)[1/p]$ and $\NK(X;\bbQ_p)$ vanish: the former by Weibel \cite[Corol.~3.3]{WMV} and the latter by the previous paragraph. It follows that \[\NK(X)\quis\NK(X;\bbZ_p),\] i.e., the spectrum $\NK(X)$ is $p$-complete. Each group $\NK_n(X)$ is therefore simultaneously derived $p$-complete and $p$-power torsion (again by Weibel), whence a result of Bhatt \cite{Bhatt_Small} states it is annihilated by a power of $p$.
\end{proof}

The following resolves Conjecture \ref{conjecture_app} assuming quasi-excellence and a version of resolution of singularities:

\begin{theorem}\label{thm2_app}
Let $X$ be a Noetherian scheme such that $X\otimes_\bbZ\bbF_p$ is quasi-excellent; assume also there exist a regular scheme $X'$ and a proper morphism $X'\to X$ inducing an isomorphism $X'\otimes_{\bbZ}\bbZ[1/p]\xrightarrow{\sim} X\otimes_{\bbZ}\bbZ[1/p]$ (in particular, $X\otimes_\bbZ\bbZ[1/p]$ is regular). Then, for any $n\in\bbZ$, the group $\NK_n(X)$ is annihilated by a power of $p$.
\end{theorem}
\begin{proof}
The cartesian square of schemes
\[\xymatrix{
X'\otimes_{\bbZ}\bbF_p\ar[r] \ar[d] &X'\ar[d]\\
X\otimes_{\bbZ}\bbF_p\ar[r]&X
}\]
is an abstract blow-up square by hypothesis, and so remains so after base changing to $\bb A_X^1$. Pro-cdh descent for algebraic $K$-theory \cite[Thm.~A]{KerzStrunkTamme_procdh} therefore implies that the bi-relative pro-abelian groups $\{\NK_n(X,X',X\otimes_{\bbZ}\bbZ/p^r\bbZ)\}_r$ vanish and that there is a resulting long exact Mayer--Vietoris sequence of pro-abelian groups
  \begin{multline}
\cdots \to \NK_n(X)\to \NK_n(X')\oplus \{\NK_n(X\otimes_{\bbZ}\bbZ/p^r\bbZ)\}_r\to\\
\to \{\NK_n(X'\otimes_{\bbZ}\bbZ/p^r\bbZ)\}_r\to\cdots. 
  \end{multline}
 The regularity of $X'$ implies that $\NK_n(X')=0$ for all $n$, and so we obtain exact sequences \[\{\NK_{n+1}(X'\otimes_{\bbZ}\bbZ/p^r\bbZ)\}_r\to \NK_n(X)\to \{\NK_n(X\otimes_{\bbZ}\bbZ/p^r\bbZ)\}_r.\] By Theorem \ref{GH_app} it is enough to prove that each of the groups $\NK_{n+1}(X'\otimes_{\bbZ}\bbZ/p^r\bbZ)$ and $\NK_n(X\otimes_{\bbZ}\bbZ/p^r\bbZ)$ is annihilated by a power of $p$. But this follows by combining Theorems \ref{GH_app} (to reduce to $r=1$) and \ref{theorem_app_EM}.
\end{proof}

For example, under the mild extra assumption that $R/p$ be quasi-excellent, the previous theorem recovers Theorem \ref{thm1_app} (take $X={\rm Spec}(R[G])$ and $X'={\rm Spec}(M_G\otimes_\bbZ R)$).

\end{document}